\documentclass{amsart}

\title{The Coding Conception of Set}

\author[Chen]{Junhong Chen}
 \address[Junhong Chen]
         {School of Mathematical Science, Fudan University, 220 Handan Road, Shanghai, 200433 China}
\email{21300180086@m.fudan.edu.cn}

\usepackage{latexsym,appendix,amsfonts,amsmath,amssymb,mathrsfs,amsthm}
\usepackage{changepage}
\usepackage{csquotes}
\usepackage{braket}
\usepackage{bbm}

\newtheorem{theorem}{Theorem}

\begin{document}

\begin{abstract}

We propose the Coding Conception of ordinals and sets, which takes Cantor’s three generating principles as its sole foundation.
Bounded sets of ordinals are generated synchronously with the ordinals themselves through a bijective encoding function that, at each stage, selects only the finitely many bounded sets actually required by the successor, limit, and restriction principles.
This selective coding yields the first-order theory $\mathsf{SC}^{\mathrm{reg}}$, which we establish is the metamathematically correct theory of the ordinals: it is bi-interpretable with $\mathsf{ZFGC}^+$, yet makes no claim about the general concept of set.
Extending the conception to full set theory via a monadic second-order ordinal theory with arithmetic and class comprehension produces two mutually inconsistent first-order set theories according to distinct maximality intuitions: a Type-A universe $\mathsf{MC}_A$, in which the power set of every ordinal is a set and the universe satisfies $\mathsf{ZFC}$; and a Type-B universe $\mathsf{MC}_B^+$, in which sets are strictly more than ordinals and a “largeness cardinal” exists, beyond which power sets remain unencodable.
We prove that this Power Set Dichotomy is unavoidable, even under potentialism, and conclude that $\mathsf{ZFC}^-+$``every cardinal has a successor'' is the only philosophically uncontroversial common fragment of any true set theory; the status of the full power-set axiom remains the sole open philosophical choice point.

\end{abstract}

\maketitle

\tableofcontents

\section{Introduction}

The concept and mathematics of sets, particularly that of finite sets, can be traced back to the most ancient origins of combinatorics, appearing almost immediately after the concept and mathematics of natural numbers.
In contrast, the concept of ordinals, as a natural extension of natural numbers, was not recognized until the work of Cantor emerged.

In the course of proving the Cantor-Bendixson theorem, Cantor realized that a set of `symbols of infinity' was necessary for this proof, and in his famous \emph{Grundlagen} \cite{Cantor1883}, he ultimately established these symbols as genuine mathematical objects: a new kind of number.
We might distill the original intention and ultimate goal of Cantor's investigation of ordinals into the following proposition.

\textbf{Cantor's Thesis on ordinals.}The construction of every mathematical object can be completed within a transfinite construction process indexed by some ordinal.

Cantor articulated an intuitive and constructive manner to state how these ordinals are generated through the following three principles.

\textbf{The first generating principle.} Every ordinal has a successor.

\textbf{The second generating principle.} Given an increasing sequence of ordinals, there is always a new ordinal that is its supremum.

\textbf{The principle of restriction or limitation.} Given an ordinal, there is always a new ordinal that is the supremum of all ordinals equipollent with the previous ordinal. For convenience, we shall henceforth call this the third generating principle.

It is therefore necessary, before proceeding, to clarify the philosophical status of Cantor's Thesis on ordinals and to what extent it faithfully reflects Cantor's original standpoint.

Let us understand this Thesis through Cantor’s conception of \emph{free mathematics}\cite{Cantor1883b}.
As Cantor emphasizes, the objects of mathematics do not originate in sensory experience, nor are they constrained by any form of spatio-temporal intuition; rather, mathematics proceeds through the formation and articulation of concepts and the relations among them
What legitimizes a mathematical concept is therefore neither its empirical origin nor its realizability in any restricted constructive sense, but its capacity to integrate coherently into the existing conceptual framework.
Once a concept is introduced in such a way that it stands in determinate relations to previously accepted concepts, it acquires what Cantor calls \emph{intrasubjective reality}, and thereby qualifies as a genuine mathematical object.
Then, due to the connectedness of the two kinds of reality, the concept would also take \emph{extrasubjective reality} into account, and we can say that the mathematical concept actually captures a mathematical object, or that it \emph{exists}.
A historical study on this topic can be found in \cite{Hallett1984}, pp.15-18.

From this perspective, mathematical construction should not be understood in the narrow sense of a step-by-step executable procedure, as in intuitionistic or constructivist frameworks.
As Cantor himself says (\cite{Cantor1883b}, p. 206, n. 6), what he proposes is a Platonic principle: the ‘creation’ of a consistent coherent concept in the human mind is actually the uncovering or discovering of a permanently and independently existing real abstract idea.
Rather, it is a \emph{free conceptual construction} governed by conditions of coherence.
The introduction of new objects is justified precisely insofar as they can be systematically incorporated into an already established network of concepts.
In this sense, Cantor’s framework is fundamentally non-constructivist, while nevertheless retaining a strong structural discipline.

We'd better keep a distance from formalism here, as the term ``coherency'' sounds like ``consistency''.
Hallett points out an argument immediately in \cite{Hallett1984}, p. 19.
\begin{quote}
    Cantor was not just or indeed very seriously occupied with the non-contradictoriness of concepts.
    (Certainly he was not concerned with formal consistency, since he had no idea of a formally presented theory.)
    Rather, Cantor’s stress on a new concept’s having ‘orderly relations to existing concepts’ is much better described as a notion of coherence, of the coherent integration into the existing conceptual framework rather than of mere non-contradiction.
    Moreover, while he mentions non-contradiction, Cantor deals with it only in so far as he attempts to rebut suggestions that the very notion of infinite number is self-contradictory.
    Outside of these negative rebuttals, which are, of course, a long way from a consistency programme, Cantor concerns himself entirely with what I call here coherence.
    We know from the development of Hilbert’s programme that to make an explicit demonstration of consistency a necessary requirement for the acceptance of a mathematical theory may be a very strong demand indeed, certainly if one insists on absolute consistency.
    But in practice mathematics has concentrated much more on something like Cantor’s coherence, the plausible extension and generalization of theories already reasonably well established.
    And in this sense it has been rather free.
\end{quote}

So, we may conclude the first part of our thesis here.

\textbf{Cantor's Thesis on Ordinals, Part A.}Every mathematical object can be completely captured within some kind of construction.

It is within this setting that the role of ordinal numbers becomes conceptually transparent.
Historically, their emergence is tied to Cantor’s analysis of transfinite iteration processes, most prominently in the Cantor-Bendixson procedure, where repeated application of the derived set operation gives rise to a hierarchy of stages extending beyond the finite.
What is crucial here is that these stages themselves form a well-ordered system, determined by a \emph{lawlike succession} (\emph{gesetzmäßige Sukzession}), which makes possible a definite counting (\emph{Abzählung}) of the process\cite{Cantor1883b}.
In this way, ordinals arise not merely as abstract objects, but as the natural indices of a structured progression of construction.

This connection is further reinforced by Cantor’s abstractionist account of order-types.
Given any simply ordered set, one may abstract from the material properties of its elements and retain only its order-structure, thereby obtaining an ``ideal paradigm'' of the set——its order type——which is itself a well-ordered system isomorphic to the original\cite{Cantor1895}.
In this sense, ordinals may be understood as canonical representatives of such order-types, functioning as idealized positions within any well-ordered progression.
Their role is thus intrinsically tied to the possibility of representing, comparing, and organizing stages of a construction process.

Consequently, the theory of ordinals embodies a general principle: mathematical objects arise through processes that can, in principle, be organized along transfinite scales, and ordinals provide the universal framework for measuring and structuring such processes.
Cantor’s three generating principles—successor, limit, and restriction—supply exactly the structural conditions required to sustain this hierarchy.
A similar study can be found at \cite{Hallett1984}, pp.142-164.
The Thesis proposed above may therefore be seen not as an external imposition, but as an explicit articulation of what is already implicit in Cantor’s framework: ordinals function as the canonical system for indexing and measuring the stages of free transfinite construction.
So, we may conclude the second part of our thesis here.
The whole thesis is thus taken as the conjunction of these two parts.

\textbf{Cantor's Thesis on Ordinals, Part B.}Every construction process can be enumerated, or counted, along some ordinal.

However, constrained by historical circumstances, Cantor's theory was not perfect.
The foremost issue was the ontological problem, which asks what ordinals truly are, since they are not as readily accepted by general intuition as natural numbers are.
Although Cantor had already studied the isomorphism classes of linearly ordered sets, i.e., the order types of linearly ordered sets, in his \emph{Beitrag} \cite{Cantor1895}, the first person to reduce ordinals to sets via the concept of order types was Russell in his \cite{Russell1903}.
There, he defined an ordinal as the order type determined by any well-ordered set.

While this provided a tentative answer to the ontological problem of ordinals, its shortcomings became apparent after the $\mathsf{ZF}$ set theory was widely considered: the equivalence class of a well-ordered set is a proper class, and the collection of all ordinals forms a proper class, far beyond the scope that axiomatic set theory is willing to entertain.

Thanks to the work of von Neumann, we now define ordinals as the most special representative of the equivalence class of well-ordered sets, namely, a transitive set on which the membership relation forms a well-order.
We can now speak of the class of all ordinals just as we speak of the class of all sets.
Cantor's three generating principles have now become internal theorems of $\mathsf{ZF}$: the first generating principle is trivial; the second generating principle, depending on a broader understanding of the term ``given,'' is best regarded as the axiom schema of replacement restricted to the class of ordinals; the third generating principle, which asserts that a cardinal successor always exists, is a significant corollary of the axiom of power set.

Nowadays, $\mathsf{ZFC}$ set theory is widely regarded as sufficient to interpret almost all modern mathematical theorems.
This means that when we claim that proving a mathematical theorem is equivalent to proving a theorem within $\mathsf{ZFC}$, we are unlikely to face technical objections.
Of course, nearly every philosophical viewpoint abhors this idea, but this observation at least allows us to provide a concrete equivalent for the broad notion of ``general mathematical proof / general mathematical object'': a set / a proof within $\mathsf{ZFC}$ set theory.

Thus, the cumulative hierarchy provides a demonstration of Cantor's Ordinal Thesis: every mathematical object / every set is generated within the transfinite construction process $\alpha \mapsto V_{\alpha}$ indexed by its $\operatorname{rank}$.
In fact, it strengthens Cantor's Ordinal Thesis, and we might rename it accordingly.

\textbf{von Neumann's Thesis on ordinals.}The construction of every mathematical object can be uniformly completed within a single transfinite construction process indexed by some ordinal.

This version of thesis on ordinals is obviously equivalent to his famous conception of set, which is often refered to as the Limitation of Size conception.
It suggest that, stronger than the axiom of choice, a set-like global well order exists, or a function between the class of all sets and the class of all ordinals exists.

At this point, two influential conceptions of set begin to diverge in their philosophical significance.

On the one hand, the cumulative hierarchy provides the standard implementation of Cantor's ordinal-based picture of set formation.
It tells us that every set appears at some level $V_\alpha$, and thus that the universe of sets can be generated through a transfinite process indexed by the ordinals.
However, this picture does not fully satisfy the demand behind Cantor's thesis.
What it gives us is, at best, a superstructure: every set is located within some $V_\alpha$, but this only tells us that the set belongs to a certain stage, not what the set is in itself.
In this sense, the iterative conception risks introducing a form of non-predicativity through the power set operation: we first obtain a larger totality and then turn back to examine its elements.
One might also express the concern in metaphysical terms: the dependence relation appears to run from a set to its supersets, rather than the other way around.
For this reason, the iterative conception, while structurally clear, does not seem to capture the explanatory core of Cantor's Thesis on ordinals.

On the other hand, von Neumann's conception, often associated with the limitation of size doctrine, suffers from a different philosophical difficulty.
It does not explain what a set is; it only tells us when something is small enough to count as one.
But this criterion is entirely external: it measures a collection against the totality of all sets or ordinals, and thereby presupposes exactly the universe whose structure it is supposed to illuminate.
In this sense, limitation of size offers not an account of sethood, but a global screening condition for sethood.
Set existence becomes dependent on a prior grasp of the whole, while the whole itself is invoked as though it were already in place.
What is left is a negative characterization of sets, not a positive explanation of their nature.
This criticism is widely known, as in Chapter 5 of \cite{Incurvati2020}.

This worry can be sharpened further.
Limitation of size also seems to undermine the positive maximality of the set-theoretic universe if we consider ordinals as a different concept from sets that could be encoded as special sets.
For if we use a comparatively well-behaved proper class such as $\mathbf{Ord}$ to regulate the size of the universe, then the resulting conception appears to constrain the universe rather than allowing it to expand.
The universe is no longer simply the fullest domain of sets, but a domain whose extension is determined by a size condition imposed from above.
This makes the conception philosophically uneasy: instead of preserving the maximality of the universe, it seems to replace that maximality with an externally imposed boundary.

All these criticism are based on our Cantor-G\"odel Platonism standpoint, as described in \cite{Ternullo2018}.
For instance, our resistance to predicativism here can be traced back entirely to G\"odel's discussion in \cite{Godel1944}.

We are thus led to a tension. The iterative conception provides a canonical generative picture, but fails to explain the nature of the objects it generates; the limitation of size conception, by contrast, replaces explanation with an external size test and threatens the maximality of the universe itself.
Neither framework, in its standard form, seems to provide a fully satisfactory account of how sets arise from the transfinite progression of the ordinals.

Accordingly, we are led to reconsider how such conceptions should be formalized.
The aim is to retain the generative role of ordinals while avoiding both the explanatory insufficiency of the cumulative hierarchy and the philosophical difficulties of the limitation of size.
In what follows, we develop an alternative approach along these lines and show how it gives rise to two natural first-order theories of sets.

\section{Coding Conception on the Theory of Ordinals}

We will show how $\mathsf{ZFGC}$ set theory arises from a process that describes how ordinals extend “as far as possible”, as $\mathsf{ZFC}$ is obviously a reduct of this theory.
However, nothing will be claimed here about the general concept of sets.

In our earlier discussion, a simple observation has already emerged: if we want to speak about the process by which ordinals are generated, we must also be able to speak about the bounded sets of ordinals that occur in that process.
For example, the second generation principle requires a strictly increasing sequence of ordinals as input in order to form its supremum; the third generation principle requires us to determine when one ordinal, regarded as a bounded set of ordinals, can be matched by a bijection with another.
Accordingly, any formal system intended to describe how ordinals extend as far as possible must contain, in addition to a sort for ordinals, a second sort for bounded sets of ordinals, together with axioms governing the interaction between the two sorts.

By contrast, unbounded sets of ordinals, or more naturally, proper classes, are not needed for the present purpose.
At no stage of the generation process do we require such large objects themselves; we only ever need their initial segments below some ordinal.
This is why the natural background logic here is a two-sorted first-order language, with one sort for ordinals and one sort for bounded sets of ordinals, linked by the membership relation $\in$.

However, if we are giving a unified transfinite account of the generation of ordinals, in which each ordinal is constructed at the stage it itself denotes, should bounded sets of ordinals not also be generated in a corresponding way?
The point is that at each stage of ordinal generation we need to refer only to one bounded sets of ordinals, provided with any possible binary encoding function that encodes a pair of ordinals as a single one(so that a function from ordinals to ordinals can be encoded as a single bounded set of ordinals, instead of a collection of ordered pairs).

For the first generating principle, the successor of an ordinal is generated from the bounded set consisting of that ordinal alone; no further coding is needed.  
The second generating principle requires, at a limit stage, an increasing cofinal sequence of ordinals; such a sequence is itself a bounded set, and the limit ordinal is generated from it.  
The third generating principle (the restriction principle) demands more: to obtain the cardinal successor \(\kappa^+\) one must appeal to a sequence that encodes, for each ordinal \(\alpha\) with \(\kappa<\alpha<\kappa^+\), a well‑ordering of \(\kappa\) of order‑type \(\alpha\). Such a sequence can be coded as a single bounded set of ordinals once a pairing function is available—a function that encodes a pair of ordinals as a single ordinal. The Gödel pairing function provides such a coding, and it is definable from ordinal addition and multiplication. Hence the most natural treatment is to incorporate ordinal addition and multiplication directly, as we shall do in Section 3. But here we could just hang the problem alone, since our theory will give us a coding function naturally as a consequence.

This suggests that bounded sets of ordinals should be generated in tandem with the ordinals themselves.
That requires a bijective function assigning to each ordinal a bounded set of ordinals since we need only construct those bounded sets that are relevant to the process of ordinal extension.
This natural and economical restriction on the sort of bounded sets is essential and useful, and we may come back to this idea later.
From this perspective, the inverse correspondence assigns to each bounded set of ordinals in the domain an ordinal marking the stage at which that set is constructed.

Taken as a whole, this correspondence can be understood as a coding of bounded sets of ordinals, much as G\"odel's $\beta$-function codes bounded sets of natural numbers by single natural numbers.
The central idea is therefore this: bounded sets of ordinals are not primitive objects standing apart from the ordinal hierarchy, but are generated and organized together with that hierarchy itself.
We therefore turn to the function from bounded sets of ordinals to ordinals that expresses this relation, and call it the encoding function.

\subsection{The theory $\mathsf{SC}$}

We now proceed to construct a formal theory that captures this conception.

Consider a two-sorted first-order language. The variables of the first sort, called \emph{ordinals}, are denoted by lowercase Greek letters $\alpha,\beta,\gamma,\dots$; the variables of the second sort, called \emph{bounded sets of ordinals}, are denoted by lowercase Latin letters $x,y,z,\dots$.
The language contains a binary relation symbol $<$ for the ordering between ordinals, a binary relation symbol $\in$ for the membership of an ordinal in a bounded set of ordinals, and a unary function symbol $e$.
The function $e$ maps bounded sets to ordinals, and it will be called the \emph{encoding function}. The theory consisting of the following axioms is denoted $\mathsf{SC}^{--}$.

\textbf{Well‑ordering.} The relation $<$ is a linear order:
$$\forall\alpha,\beta\;(\alpha<\beta\;\lor\;\beta<\alpha\;\lor\;\alpha=\beta)$$
$$\forall\alpha,\beta\;(\alpha<\beta\to\alpha\neq\beta)$$
$$\forall\alpha,\beta,\gamma\;(\alpha<\beta\;\land\;\beta<\gamma\to\alpha<\gamma)$$
and every non‑empty bounded set of ordinals has a least element:
$$\forall x\,\bigl(\exists\alpha\in x\to\exists\alpha\in x\;\forall\beta\in x\;(\neg(\beta<\alpha))\bigr)$$

\textbf{Replacement Schema.} For every formula $\varphi(\gamma,\delta)$ (possibly with parameters), we have a replacement axiom:
$$\forall\alpha\,\Bigl(\forall\gamma<\alpha\;\exists!\delta\;\varphi(\gamma,\delta)\;\to\;\exists\beta\;\forall\gamma<\alpha\;\exists\delta<\beta\;\varphi(\gamma,\delta)\Bigr)$$

\textbf{Extensionality and Boundedness.} Bounded sets of ordinals are extensional and bounded:
$$\forall x\forall y\,\bigl(\forall\alpha\,(\alpha\in x\leftrightarrow\alpha\in y)\to x=y\bigr),\qquad
\forall x\exists \alpha\forall\beta\in x(\beta<\alpha)$$

\textbf{Bounded Comprehension Schema.}
For every formula $\varphi(\beta)$ (possibly with parameters), we have a bounded comprehension axiom:
$$\forall\alpha\;\exists x\;\forall\beta<\alpha\;\bigl(\beta\in x\leftrightarrow\varphi(\beta)\bigr)$$

\textbf{Encoding.} The function $e$ is a bijection:
$$\forall \alpha\exists x(\alpha=e(x)),\qquad
\forall x\forall y\;(e(x)=e(y)\to x=y).$$

\textbf{Order Preservation.} The encoding function respects the order:
$$\forall\alpha\forall x\;(\alpha\in x\to\alpha<e(x))$$

Adding the following axiom gives $\mathsf{SC}^-$.
If we add its negation, the extended theory will be denoted by $\mathsf{SC}^{\text{fin}}$.

\textbf{Infinity.} $\omega$ exists:
$$\exists \alpha>0\forall\beta<\alpha\exists\gamma<\alpha(\beta<\gamma)$$

The theorem below is immediate.

\begin{theorem}
Define $\alpha E\beta$ in $\mathsf{SC}^{-}$ as $\exists x(\alpha\in x\wedge \beta=e(x))$. Then, the structure consisting of all ordinals together with the relation $E$ is a model of $\mathsf{ZFC}^-$, obtained by omitting the power set axiom from $\mathsf{ZFC}$ and equipping the Collection Schema instead of the Replacement Schema.
\cite{Gitman2016}
In fact, if we also take $<$ into consideration, we get a model of $\mathsf{ZFGC}^-$, where additional axioms about $<$ state that $<$ plays the role of a set-like global well-ordering that respects $<$.
\end{theorem}
\begin{proof}
Suppose $(Ord,Set^*;<,\in,e)\vDash\mathsf{SC}^-$; we'll prove that $(Ord,E,<)\vDash\mathsf{ZFGC}^-$.
To prove extensionality, notice that $\forall\gamma(\gamma E\alpha\leftrightarrow\gamma E\beta)$ implies $\exists x,y(\forall\gamma(\gamma\in x\leftrightarrow \gamma\in y)\wedge e(x)=\alpha\wedge e(y)=\beta)$; the extensionality of bounded sets of ordinals shows that $x=y$, so $\alpha=\beta=e(x)=e(y)$.
Notice that order preservation proves $\alpha E\beta\to\alpha<\beta$; the axiom of pairing, union, and the separation schema can also be derived easily from the bounded comprehension schema.
Since $<$ is a well-ordering, for every $\alpha=e(x)$, we have a $<$ minimum $\beta\in x$, which also plays the role of the $\in^*$ minimal element of $\alpha$, proving the axiom of foundation.
Notice that we didn't use the replacement schema and the axiom of infinity in the previous proof.
The replacement schema in $\mathsf{ZFGC}^-$ is just an immediate consequence of the replacement schema in $\mathsf{SC}^-$.
As for the axiom of infinity in $\mathsf{ZFGC}^-$, first we can get a least infinite limit ordinal from the axiom of infinity in $\mathsf{SC}^-$; now we have a bounded set consists of all finite ordinals from the bounded comprehension schema, and the encoding of this set will be interpreted as the set of all natural numbers in set theory.

However, we want Collection in $\mathsf{ZFGC}^-$ instead of Replacement, so let's prove a stronger result.
\begin{theorem}
For every formula $\varphi(\gamma,\delta)$ (possibly with parameters),
$$\mathsf{SC}^-\vdash\forall\alpha\,\Bigl(\forall\gamma<\alpha\;\exists\delta\;\varphi(\gamma,\delta)\;\to\;\exists\beta\;\forall\gamma<\alpha\;\exists\delta<\beta\;\varphi(\gamma,\delta)\Bigr).$$
\end{theorem}
\begin{proof}
    Let $\psi(\gamma,\delta)=\varphi(\gamma,\delta)\land\forall\delta'<\delta(\neg\varphi(\gamma,\delta'))$; we'll show that $\forall\gamma<\alpha\exists !\delta(\psi(\gamma,\delta))$ from $\forall\gamma<\alpha\exists\delta(\varphi(\gamma,\delta))$.
    Given $\gamma<\alpha$, suppose $\delta_0$ is an ordinal such that $\varphi(\gamma,\delta_0)$ holds; by bounded comprehension, we may form $x=\{\delta<\delta_0\mid\varphi(\gamma,\delta)\}$.
    If $x=\varnothing$, $\delta=\delta_0$ is what we need; otherwise, pick the $<$ minimum element in $x$ as our $\delta$.
\end{proof}
The rest of this proof is obvious.
\end{proof}
\begin{theorem}
$\mathsf{SC}^-$ is in bi-interpretation with $\mathsf{ZFGC}^-$.
\end{theorem}
\begin{proof}
We'll show that, after giving $(V,\in,<)\vDash\mathsf{ZFGC}^-$, the identity function $\operatorname{id}:V\to V$ is just enough for having $(V,V;<,\in,\operatorname{id})\vDash\mathsf{SC}^-$.
Since most axioms can be checked directly, we'll focus on the replacement schema.
Given $x\in V$, the set-likeness of $<$ proves $\operatorname{Ext}_<(x)=\{y\in V\mid y<x\}\in V$.
Now, if $\forall z<x\exists! w(\varphi(z,w))$ we have $\forall z\in\operatorname{Ext}_<(x)\exists! w(\varphi(z,w))$; so the replacement schema in $\mathsf{ZFGC}^-$ proves $\exists y\forall z\in\operatorname{Ext}_<(x)\exists w\in y(\varphi(z,w))$, and $w\in y\to w<y$ shows $\forall z<x\exists w<y(\varphi(z,w))$.
\end{proof}

$\mathsf{SC}^{\text{fin}}$ is of little interest to us because it is not difficult to show by the same argument that it amounts to $\mathsf{PA}$ equipped with a coding function for finite sets of natural numbers; however, since we have a computable coding function for finite sets of natural numbers, it is mutually interpretable with $\mathsf{PA}$.

When we speak of a \emph{set}, we mean an ordinal $\alpha$ equipped with the relation $E$ on it; when we need to regard $\mathsf{SC}^-$ and its various strengthenings as set theories, we directly use lowercase Latin letters $x, y, z, \dots$ as variable symbols.
We shall call those ordinals that remain ordinals when viewed as sets \emph{set ordinals}.
There is an obvious isomorphism that sends each ordinal $\alpha$ to the corresponding set ordinal $\check{\alpha}$; in $\mathsf{ZFGC}^-$, the existence of this function is precisely equivalent to a bijection $o: \mathbf{V} \to \mathbf{Ord}$ with the property that $o(x) = \operatorname{otp}_{<} \{ y \mid y < x \}$ for some global well-ordering $<$.
In fact, what we discuss in the set-theoretic sense as $x \mapsto \operatorname{Ext}_{<}(x) = \{ y \mid y < x \}$ is exactly the composite map $\alpha \mapsto \operatorname{Ext}_{<}(\alpha) \mapsto e(\operatorname{Ext}_{<}(\alpha))$ obtained in $\mathsf{SC}^-$, where $\operatorname{Ext}_{<}(\alpha)$ is precisely that bounded set of ordinals formed by all ordinals $\beta < \alpha$ in $\mathsf{SC}^-$.

A global well-ordering is powerful: although $\mathsf{ZFC}^-$ using Replacement cannot prove $\mathsf{ZFC}^-$ using Collection, the above proof shows that $\mathsf{ZFGC}^-$ using Replacement and $\mathsf{ZFGC}^-$ using Collection are the same theory; in fact, we can even prove the Reflection Principle.

\begin{theorem}
    $\mathsf{ZFGC}^-$ proves that, for every formula $\varphi(\vec{x},\vec{p})$ with parameters, there is a transitive set $t$ such that $\vec{p}\subseteq t$, and for every $\vec{x}\subseteq t$, we have $\varphi(\vec{x},\vec{p})\leftrightarrow\varphi^t(\vec{x},\vec{p})$.
\end{theorem}
\begin{proof}
    Let's prove a strengthened version of this claim: this transitive set $t$ is actually of the form $\{x\mid x<s\}$, and $s$ has the property $\forall x<s\exists y<s(x<y)$.
    Consider any construction sequence of $\varphi$, namely $\varphi_0(\vec{x}_0),\dots,\varphi_n(\vec{x}_n)=\varphi$, such that $\vec{x}_n\subseteq\dots\subseteq\vec{x}_0$; the free variables of $\varphi_i(i\leq n)$ are all in $\vec{x}_i$, and every variable in $\vec{x}_i$ is either free or does not occur in $\varphi_i$.
    Without loss of generality, we can suppose that every $\varphi_i$ is either an atomic formula, the negation of a previous formula, the conjunction of two previous formulae, or the existential quantification of a previous formula.
    We define a function $f_i(\vec{x}_i)$ for every $\varphi_i$: if $\varphi_i$ is not the existential quantification of a previous formula, set $f_i(\vec{x}_i)\equiv \varnothing$; if $\varphi_i=\exists x_{n_i}(\varphi_j)$, $f_i(\vec{x}_i)$ is the $<$-least $s$ such that, if $\exists x_{n_i}(\varphi_j)$ is correct, then the $<-$least such $x_{n_i}$ is $<s$.
    Now we define $F(s)$ to be the strict $<$-supremum of the set $\{f_i(\vec{x}_i)\mid i\leq n,\forall x_{i_k}\in\vec{x}_i(x_{i_k}<s)\}$; take $s_0$ to be the strict $<$-supremum of $\vec{p}$, and inductively take $s_{n+1}=F(s_n)$.
    It's easy to observe that the $t=\{x\mid x<s\}$ obtained by $s=\operatorname{ssup}_{<}\{s_n\mid n\in\omega\}$ is the desired transitive set.
\end{proof}

Now we add a new axiom and name the extended theory $\mathsf{SC}$.

\textbf{Limitlessness in Size.} Every cardinal has its successor in the sense of $\mathsf{ZFGC}^-$.

If we replace the additional axiom above with the following one, the theory obtained will be called $\mathsf{SC}^+$.

\textbf{Uniform Boundedness.} The encoding function is uniformly bounded:
$$\forall\alpha\;\exists\beta\;\forall x\,\bigl(\forall\gamma\in x\;(\gamma<\alpha)\to e(x)<\beta\bigr)$$

There is a natural strengthening of the above axiom, which will lead to $\mathsf{SC}^{\text{reg}}$.

\textbf{Regularity of the Encoding.} $e(x)<e(y)$ if $\operatorname{sup}(x)<\operatorname{sup}(y)$.

The following theorem is obvious.
\begin{theorem}
    $\mathsf{SC}^+$ bi-interprets $\mathsf{ZFGC}$, and $\mathsf{SC}$ bi-interprets $\mathsf{ZFGC}^{-}+$every cardinal has its successor.
    $\mathsf{SC}^{\text{reg}}$ is bi-interpretable with the theory $\mathsf{ZFGC}+<$ respects the partial order induced by $\operatorname{rank}$, which will be denoted by $\mathsf{ZFGC}^+$.
\end{theorem}
\begin{proof}
    $\mathsf{SC}$ bi-interprets $\mathsf{ZFGC}^{-}+$every cardinal has its successor is obvious, so let's look at $\mathsf{SC}^+$.

    If the encoding is regular, we use induction on all set ordinals to show that $<$ respects the partial order induced by $\operatorname{rank}_E$.
    For any set ordinal $\alpha$, the induction hypothesis states that if $\gamma<\beta<\alpha$ are set ordinals and $\delta,\xi$ are ordinals such that $\operatorname{rank}_E(\delta)=\gamma$, $\operatorname{rank}_E(\xi)=\beta$, then $\delta<\xi$.
    Our goal is to prove that, if $\beta<\alpha$ is a set ordinal and $\gamma,\delta$ are ordinals such that $\operatorname{rank}_E(\gamma)=\beta$ and $\operatorname{rank}_E(\delta)=\alpha$, then $\gamma<\delta$.
    Assume $\gamma=e(y)$ and $\delta=e(x)$, we have $z=\{\xi\in x\mid \operatorname{rank}_E(\xi)\leq\beta\}\subseteq x$.
    Obviously $\operatorname{rank}_E(e(z))=\beta+_E1\leq\alpha$($\beta+_E1$ is the next set ordinal, since we have $\beta+1$ as the next ordinal); if $\beta+_E1<\alpha$, we have $\beta<\beta+_E1<\alpha$, so induction hypothesis indicates that $\gamma<e(z)$, and we will prove that $\operatorname{sup}(z)<\operatorname{sup}(x)$, so $e(z)<\delta$ follows from regularity.
    If not, since $z\subsetneq x$, there must be some $\zeta_1\in x\backslash z$ and $\zeta_2\in z$ such that $\operatorname{rank}_E(\zeta_1)=\beta+_E1>\beta\geq\operatorname{rank}_E(\zeta_2)$ but $\zeta_1<\zeta_2$, but the induction hypothesis shows that this case is impossible.

    So the rest is to prove that, if $\alpha=\beta+_E1$, we still have $\gamma<\delta$.
    This time we will directly show that $\operatorname{sup}(y)<\operatorname{sup}(x)$: if not, there still will be some $\zeta_1\in x$ and $\zeta_2\in y$ such that $\operatorname{rank}_E(\zeta_1)=\beta>\operatorname{rank}_E(\zeta_2)$ but $\zeta_1>\zeta_2$; notice that $\operatorname{rank}_E(\zeta_2)<\beta<\alpha$, the induction hypothesis again gives us the desired contradiction.
\end{proof}

\subsection{Comments on $\mathsf{SC}$}

Here are more comments on the previous formalization of our conception.

Few objections can be raised against the requirements of well-ordering, extensionality for sets, and the axiom of infinity.

The boundedness of sets, the bounded comprehension schema, and order preservation are all directly mandated by our conception: we do not directly engage with the question "why all sets cannot form a set," but rather point out that such sets are unnecessary—and thus need not be studied—in the process of extending the ordinals; moreover, we naturally desire a simple predicative requirement, namely that every element of a set must be constructed before the set itself is constructed, which is precisely what order preservation expresses.

We now turn to the Replacement schema.
It precisely formalizes the second generating principle of Cantor that we discussed in the previous section.
An alternative justification for the Replacement schema is to justify the Reflection principle, since they are obviously provably equivalent over the rest of $\mathsf{SC}^-$.
We do not take this argument here because if we adopt the Reflection principle as revealing the indescribability of the whole universe of ordinals, we are just taking the universe as a completed whole a priori, which contradicts our original conception.
We may come back to this issue in the next section.

The issues concerning Limitlessness in Size and Uniform Boundedness are far more serious.
The former is a direct formalization of the third generating principle, requiring an encoding function only to interpret what a bijection is, whereas the latter is evidently a characterization of a specific property of the encoding function: that an enumeration of all subsets of a given fixed set ordinal can always be completed.
In fact, as long as one can traverse all well-orderings that are binary relations on a given set ordinal—which, via canonical coding techniques, can be reduced to traversing all subsets satisfying a certain formula—then Limitlessness in Size can be proved from the axiom schema of replacement.
There is no strong reason to accept that subsets satisfying that formula are enumerable while all subsets are not, for both attempts aim to create new set ordinals that are essentially larger, and we know nothing about the cardinalities of the two newly created set ordinals; that belongs to the realm of the generalized continuum hypothesis.

On the other hand, rejecting Uniform Boundedness leads to strange consequences: at some least ordinal $\alpha$, we suddenly discover that we must enumerate unboundedly many subsets of $\alpha$ to let the ordinals extend further, yet for every ordinal below $\alpha$, all its subsets are completely enumerated below some ordinal.

Another consideration is that if we have a model of $\mathsf{ZFGC}$, then by first comparing ranks and then comparing the existing well-ordering, we can adjust it to a model of $\mathsf{ZFGC}^+$.
This indicates that Uniform Boundedness is merely an innocuous weakening of regularity.
Regularity itself has its own rationale: it requires us to enumerate bounded sets of ordinals always in the natural order of their suprema, and it forbids enumerating some small bounded set of ordinals too far beyond.

At any stage of generating a limit ordinal $\alpha$, we only care about the bounded sets strictly necessary to propel the ordinal construction forward.
For the second generating principle, the required increasing sequences are intrinsically bounded by $\alpha$.
For the third generating principle, establishing the absence of bijections relies entirely on bounded sets whose suprema are similarly bounded by $\alpha$.
On the other hand, to generate ordinals beyond $\alpha$, the suprema of the necessary newly formed bounded sets will naturally exceed $\alpha$ itself.
This indicates that the bounded sets required below $\alpha$ are, by the very design of the generating principles, fully exhausted at this stage.
So it could be enumerated completely, and there are no obstacles now keeping us from accepting the Regularity.

Since we are dealing with a constructive process that is indifferent to extraneous sets, we are fully justified in demanding that this exhaustion occurs in the most docile and orderly manner possible.
This methodological demand for tameness is precisely what the Regularity axiom—or Uniform Boundedness—formalizes.

Therefore, the intuition behind regularity is quite clear, and we have good reason to accept it directly.
In fact, our conclusion may be read as follows.

\begin{center}
\fbox{
  \parbox{0.5\textwidth}{
    \textbf{$\mathsf{SC}^{\text{reg}}$ is the (metamathematically) \\
    true first order theory for ordinals.}
  }
}
\end{center}

Finally, why can we not defend the claim $\exists p \, (V = \mathrm{HOD}_p)$?
In the context of $\mathsf{SC}$, this is equivalent to requiring that $<$ be definable in terms of $E$.
However, it is the relation $<$—rather than $E$—that occupies a more foundational position in the construction.
Thus, it seems unreasonable to define the pre-existing ordering $<$ using the subsequently introduced relation $E$.
Moreover, a straightforward argument shows that there can be no relation $E$ definable solely in terms of $<$ that satisfies the required axioms: if so, $\mathsf{ZFC}$ would prove that the universe has a definable global well-ordering, contradicting the fact that it does not exist after an Easton forcing over $\mathbf{L}$ that adds two Cohen subsets to every regular cardinal.
Consequently, there is no justification for imposing any definability connection between them.

\section{Coding Conception on the Theory of Sets}

Our previous discussion has indicated that $\mathsf{ZFC}$ can be regarded as an equiconsistent corollary theory that emerges when one demands a theory describing the sufficiently extended process of ordinals.
However, if we aim to obtain a theory that purports to describe the concept of ``set'', then all the previous variants of $\mathsf{SC}^-$ are inadequate since none of their axioms derive from any intuitive principle that we expect to characterize the properties that all sets should satisfy.
How, then, can we obtain a satisfactory pure set theory?

A well-known fact is that ``every set can be encoded as a bounded set of ordinals'' is itself another formulation of the Axiom of Choice.
Yet our theory can only handle bounded sets of ordinals, so we need to defend this principle first.
We adopt a form of Cantor's own defense: applying Cantor's Ordinal Thesis directly to sets.
To this end, we must first clarify what can be achieved in a single-step mathematical construction, especially what do we mean by a single-step construction of set, thereby delineating which mathematical constructions must be regarded as multi-step or even transfinite constructions.
The simplest understanding of this question still follows Gödel’s view: sets are always generated by the set-of operation.
In the present context, this is interpreted as meaning that a single-step construction of a set is just a single-step set-of operation.
Therefore, if we further apply Cantor’s Ordinal Thesis, this requires that there exist a well-ordering $e$ on the canonical graph $(\operatorname{TCl}(\{x\}),\in)$ of every set, such that $(y,z)\in e$ whenever $y\in z\in\operatorname{TCl}(\{x\})$.
This, of course, implies the existence of a well-ordering on $x$.

The simplest way to consider all bounded sets of ordinals is, of course, to consider all $\mathsf{SC}$ models in which the class of all possible ordinals is taken as the true class of ordinals, and to assert that a set exists if and only if it is a bounded set of ordinals identified in some $\mathsf{SC}$ model.
However, enumerating these $\mathsf{SC}$ models is rather inconvenient; a direct simplification is to enumerate all binary relations $E$ on ordinals that yield models of $\mathsf{ZFGC}$. Previous work has already shown that this involves no technical difficulty.
Given that the most commonly used system of second-order logic, namely monadic second-order logic, allows us to enumerate classes over ordinals rather than binary relations, we naturally wish to interpret a class over some ordinal as a binary relation on ordinals via a uniform coding method.
It is well known that there is a canonical well-order on ordered pairs of ordinals: we first compare the larger element of each pair and then proceed lexicographically; we may still use $<$ to denote this well-order.
Even a very weak set theory suffices to prove the existence of an order isomorphism that maps each ordered pair of ordinals to an ordinal, which is exactly the order type of the set of all ordered pairs of ordinals smaller than this pair under the aforementioned well-order.
We naturally regard a function taking ordered pairs as input as a binary function and call this function the Gödel pairing function, denoted $\operatorname{Pair}(\alpha,\beta)$.

In a set-theoretic system, we can simply prove the existence of this function, but doing so raises the suspicion that the following issue might arise: for the same pair $(\alpha,\beta)$, switching the relation $E$ could result in $\operatorname{Pair}(\alpha,\beta)$ yielding two distinct ordinals (since $E$ is used in its definition).
To avoid this, we desire a better definition of $\operatorname{Pair}$ such that its value is independent of the choice of the specific relation $E$.
Since ordinals are the natural extension of natural numbers, and the archetypal theory of natural numbers is arithmetic, the most natural approach is to add axioms of ordinal arithmetic.
By the results in Appendix A, we can indeed define the required $\operatorname{Pair}$ on the class of ordinals using only $<, +, \times$.

Adding ordinal arithmetic has another benefit: clearly $\omega$ is definable, which means that we can easily obtain the arithmetic structure on natural numbers by restricting the domain of discourse.
To enumerate all classes that decode to models of $\mathsf{ZFGC}$, we need a first-order truth predicate with second-order parameters to determine this, since $\mathsf{ZFGC}$ is clearly not finitely axiomatizable. For this purpose, the most natural approach is to directly define the desired truth predicate via the class comprehension principle and the arithmetic structure of natural numbers in the standard way; readers who doubt the technical details of this may refer to Appendix B.

In summary, we have now determined all the non-logical symbols that the intended monadic second-order logic should possess: the relations $<, +, \times$ on ordinals and the membership relation $\in$ from ordinals to classes are entirely sufficient.
We continue to use $\alpha,\beta,\dots$ as free variables for ordinals, $X,Y,\dots$ as free variables for classes, and $x,y,\dots$ as free variables for bounded classes of ordinals.

We now list the axioms that are intuitively valid in this logic.
First and foremost are the axioms describing ordinal arithmetic; we denote the theory formed by these axioms as $\mathsf{OA}$.

\textbf{Infinite Well‑ordering; the first generation principle.} On $<$:
$$\forall\alpha,\beta\;(\alpha<\beta\;\lor\;\beta<\alpha\;\lor\;\alpha=\beta)$$
$$\forall\alpha,\beta\;(\alpha<\beta\to\alpha\neq\beta)$$
$$\forall\alpha,\beta,\gamma\;(\alpha<\beta\;\land\;\beta<\gamma\to\alpha<\gamma)$$
$$\exists \alpha(\exists\beta(\beta<\alpha)\wedge\forall\beta<\alpha\exists\gamma<\alpha(\beta<\gamma))$$
$$\forall X\,\bigl(\exists\alpha\in x\to\exists\alpha\in x\;\forall\beta\in x\;(\neg(\beta<\alpha))\bigr)$$

\textbf{Replacement Schema; the second generation principle.} For every formula $\varphi(\gamma,\delta)$ (possibly with both first order and second order parameters), we have a replacement axiom:
$$\forall\alpha\,\Bigl(\forall\gamma<\alpha\;\exists!\delta\;\varphi(\gamma,\delta)\;\to\;\exists\beta\;\forall\gamma<\alpha\;\exists\delta<\beta\;\varphi(\gamma,\delta)\Bigr)$$

\textbf{Extensionality.} Classes are extensional:
$$\forall X\forall Y\,\bigl(\forall\alpha\,(\alpha\in X\leftrightarrow\alpha\in Y)\to X=Y\bigr)$$

\textbf{Comprehension Schema.}
For every formula $\varphi(\alpha)$ (possibly with both first order and second order parameters), we have a comprehension axiom:
$$\exists X\;\forall\alpha\;\bigl(\alpha\in X\leftrightarrow\varphi(\alpha)\bigr)$$

\textbf{Arithmetic.} $+$ and $\times$ have been computed recursively:
$$\forall\alpha(\alpha+0=\alpha\land \alpha\times 0=0)$$
$$\forall\alpha,\beta(\alpha+\beta=\operatorname{ssup}\{\alpha+\gamma\mid\gamma<\beta\})$$
$$\forall\alpha,\beta(\alpha\times\beta=\operatorname{ssup}\{\alpha\times\gamma+\delta\mid\gamma<\beta\land\delta<\alpha\})$$

It is hard to claim that the third generation principle is a principle of ordinal arithmetic, for it actually concerns the successor of cardinals, which ought to be a principle of cardinal arithmetic.
Even if we establish that cardinality is a concept subordinate to ordinality, cardinal arithmetic bears no relation to recursive notions in any sense; it is more like a relation that must be studied relying on set theory.
For instance, cardinal addition arises from disjoint unions of sets, cardinal multiplication from Cartesian products of sets, cardinal exponentiation from the power-set operation on sets, and so on.
Accordingly, we call the theory obtained by further adding the following axioms to $\mathsf{OA}$ the theory $\mathsf{OA}^+$.

\textbf{The third generation principle}
For any ordinal $\beta$, there exists some $\alpha$ such that no class $X$ codes a bijection between $\alpha$ and $\beta$ in the sense of $\operatorname{Pair}$.

$\mathsf{OA}$ can discuss functions on ordinals via $\operatorname{Pair}$, so it could prove those elementary theorems on ordinal arithmetic, such as the associativity of addition and multiplication, the commutativity of addition and multiplication on natural numbers, and the existence of the recursively-defined ordinal exponential function, and so on.

However, the most crucial topic here is to define the canonical encoding that depicts the canonical graph of a set as a bounded class of ordinals.
By a well-known result in $\mathsf{ZF}$ set theory, an accessible pointed graph $(g,e)$ is an exact picture if and only if it is extensional and well-founded; in $\mathsf{ZFC}$, we can assume without loss of generality that $G$ is an ordinal.
Now, since $e$ is a binary relation on $g=\alpha$, and every point is on some edge, we can take $g=\operatorname{dom}(e)\cup\operatorname{ran}(e)$, so we only need a predicate that says ``$e$ gives an extensional well-founded apg on some ordinal''; this will be denoted by $\operatorname{isPreSet}(e)$.

It is easy to define when two presets $e_1,e_2$ represent the same set and when one preset represents an element of the other.
However, the real obstacle here is that we cannot take a set to be the third-order class of all its representational presets, and we do not know whether there is a canonical well-ordering on all bounded classes of ordinals.
If there is such a well-ordering, then by the same argument as in $\mathsf{SC}$, we will have a maximal encoding of all sets, and the universe will be a model of $\mathsf{ZFGC}$; if there is no such well-ordering, then there is no maximal encoding of all sets, and we may face more difficulties in deciding the true theory of all sets.

Nevertheless, this difficulty does not prevent us from deriving a first-order set theory from the second-order ordinal theory: we need only interpret quantification over sets as quantification over presets and interpret the equality and membership relations between sets correspondingly as the relations of representing the same set and of membership between presets.
However, it's easy to observe that the classes in second order ordinal theory can only be viewed as those set-classes injectively mapped into $\mathbf{Ord}$, which means we are not able to derive a second order set theory directly.
Anyway, since we have previously argued that every set ought to be represented by some preset, we arrive at the following conclusion.

\begin{center}
\fbox{
  \parbox{0.65\textwidth}{
    \textbf{The true first-order set theory is a consequence\\ of the true second-order ordinal theory.}
  }
}
\end{center}

A minor issue is that the true first-order ordinal theory $\mathsf{SC}^+$ (i.e., $\mathsf{ZFGC}^+$) we established earlier requires, in addition to the $<$ relation, a global $\in$ relation (namely, the predicate $E$ on ordinals corresponding to the encoding function).
However, the true second-order ordinal theory we have established here does not accommodate this predicate. This is because we can now freely quantify over and discuss second-order objects on ordinals (such as this predicate), and thus we only need to add the following axiom.

\textbf{$\mathsf{MU}$, Maximality of the Universe of Sets.} For every bounded class $g$, there is a class $E$ such that $\exists \alpha\forall\beta(\beta\in g\leftrightarrow\operatorname{Pair}(\beta,\alpha)\in E)$, and we have $\operatorname{isRegEncoding}(E)$ at the same time, so the axiom might be read as ``every set can be encoded''.

$\mathsf{ZFGC}^-$ provides a bijection between sets and (set) ordinals that respects $<$.
Therefore, if some preset $g$ is coded in $E$, i.e, $\exists\alpha\forall\beta(\beta\in g\leftrightarrow\operatorname{Pair}(\beta,\alpha)\in E)$ and $\operatorname{isSemiEncoding}(E)$, then we can decode $\alpha$ as $g$ in the set universe represented by $E$ and draw the conclusion that the set represented by $g$ is actually inside this universe.
On the other hand, every bounded class $e$ is, in fact, a set, so there is always a preset $g$ representing it; furthermore, we have $e$ coded in $E$ if and only if $g$ is coded in $E$.
As a conclusion, a semi-encoding encodes all (pre)sets if and only if it encodes all bounded classes.

Now, how can we talk about collections of presets?
An easy way is to introduce the notion of hyperclass.
In set theory, a hyperclass is a collection of classes, while a class is a collection of sets.
Here we may still take hyperclasses as collections of classes, but now classes are collections of ordinals.
We define a hyperclass as a formula with possibly parameters that has only one variable for classes.
A hyperclass is uniformly bounded if and only if, for some ordinal $\alpha$, every class that satisfies the formula is bounded by $\alpha$.

Now we want to talk about hyperclass in second order theory itself, which leads to the following definition.
A class $C$ encodes a hyperclass if and only if for every class-element $D$ of the hyperclass, there exists precisely one ordinal $\alpha$ such that for every $\theta\in C$, $\theta=\operatorname{Pair}(\alpha,\beta)$ for some $\beta$ if and only if $\beta\in D$.
For any class $C$, the class $\{\beta\mid\operatorname{Pair}(\alpha,\beta)\in C\}$ will be denoted as $C_\alpha$: the $\alpha-$th section of $C$.
It is easy to observe that, a hyperclass is encoded as a bounded class if and only if the hyperclass itself is uniformly bounded, and a preset represents this hyperclass as a set of sets of ordinals.

\subsection{Type-A Universe: $\mathsf{MC}_A$}

The claim that there exists a global well-ordering in the universe of sets can be deduced from, in our present context, the requirement that there exists a Regular Encoding in which all bounded classes are encoded.
In its more traditional equivalent form, the latter one is equivalent to the claim that the universe of all sets indeed satisfies $\mathsf{ZFGC}$.

We will find that an extremely weak form of commitment to the power set axiom almost entails this remarkably strong assumption, which is incompatible with another natural perspective on the maximality of the set-theoretic universe.
This weak formulation of the power set axiom within second-order ordinal theory is stated as follows.

\textbf{$\mathsf{UE}$, Uniform Encodability.} Every uniformly bounded hyperclass can be encoded as a single bounded class.

As the power ``set'' of any ordinal is obviously a hyperclass, uniform encodability is equivalent to the statement that, for every ordinal $\alpha$, there is a bounded class that encodes $P(\alpha)$; if we take a regular encoding for this bounded class, then a trivial argument shows that this regular encoding encodes all bounded classes that are bounded by $\alpha$.
The underlying idea is straightforward: we do not presuppose the existence of a single, global regular encoding that encodes all bounded classes.
Instead, we conclude that for every ordinal $\alpha$, there exists a sufficiently well-behaved regular encoding that ensures all bounded classes below $\alpha$ are encoded.
Intuitively, we may say that this regular encoding correctly computes the power set of each ordinal up to $\alpha$.

\begin{theorem}
    $\mathsf{ZFC}$ proves that every standard model of $\mathsf{OA}$ extended with $\mathsf{UE}$ also satisfies the axiom ``there exists a maximal regular encoding of all bounded classes'': $\exists E(\operatorname{isRegEncoding}(E)\land \forall g\exists\alpha\forall\beta(\beta\in g\leftrightarrow\operatorname{Pair}(\beta,\alpha)\in E)$.
\end{theorem}
\begin{proof}
    We can prove that these standard models are just $(\kappa,P(\kappa))$ for some inaccessible cardinal $\kappa$: Uniform Encodability shows that $\kappa$ is a strong limit, while the Replacement Schema shows that $\kappa$ is regular.
    Since $\kappa$ is infinite, we finally conclude that $\kappa$ is inaccessible.
    Now $V_\kappa=H_\kappa$, and every set in $H_\kappa$ is represented by some preset in $P(\kappa)$.
    However, there is obviously a well-ordering on $V_\kappa$ that respects $\operatorname{rank}$ and has order type $\kappa$.
    Therefore, a regular encoding derived from this well-ordering is an element of $P(\kappa)$, which is just what we want.
\end{proof}

In fact, we can get rid of the standard semantics after embracing the more powerful choice-like axiom below.

\textbf{$\mathsf{TDC}$, Transfinite Dependent Choice.} Suppose that, for every ordinal $\alpha$ and an $\alpha$-sequence of classes $\langle P^\beta\rangle_{\beta<\alpha}$(which can be treated as a single class with $P^\beta$ being the section of it at $\beta$), there exists a class $P^\alpha$ such that $\varphi(\alpha,P^\alpha,\langle P^\beta\rangle_{\beta<\alpha})$ if and only if, for every $\beta<\alpha$, we have $\varphi(\beta,P^\beta,\langle P^\gamma\rangle_{\gamma<\beta})$.
Then, there exists a class $P$ such that for every $\alpha$ we have $\varphi(\alpha,P_\alpha,\langle P_\beta\rangle_{\beta<\alpha})$.

This axiom seems natural: it provides an $\mathbf{Ord}$-length sequence for a dependent choice process.
Another similar one is the axiom of \textbf{$\mathsf{CC}$, Class Choice}: if for every $\alpha$ there exists some class $C$ such that $\psi(\alpha,C)$, then there exists a globally $C$ such that we have $\psi(\alpha,C_\alpha)$ for every $\alpha$.
The main difference between these two axioms is that, in the axiom of Class Choice, the $C_\alpha$ is independent of what choice we have already made; but in Transfinite Dependent Choice, we can take them as a parameter $\langle P^\beta\rangle_{\beta<\alpha}$ in our formula.

\begin{theorem}
     $\mathsf{OA}\vdash\mathsf{TDC}\to\mathsf{CC}$.
\end{theorem}
\begin{proof}
    Suppose for every $\alpha$ there exists some class $C$ such that $\psi(\alpha,C)$.
    Let $\varphi(\alpha,C^\alpha,\langle C^\beta\rangle_{\beta<\alpha})$ be the conjunction of $\psi(\alpha,C^\alpha)$ and $\forall\beta<\alpha(\psi(\beta,C^\beta))$; it is easy to verify that it satisfies the condition for applying Transfinite Dependent Choice.
    Now the class obtained from Transfinite Dependent Choice is exactly what we need for Class Choice.
\end{proof}

Now we can eliminate standard semantics from the previous theorem.

\begin{theorem}
    The theory $\mathsf{OA}+\mathsf{MU}+\mathsf{TDC}+\mathsf{UE}$ proves that there exists a maximal regular encoding of all bounded classes.
    We may denote this theory as $\mathsf{MC}_A$.
\end{theorem}
\begin{proof}
    We will show that the property $\varphi(\alpha,P^\alpha,\langle P^\beta\rangle_{\beta<\alpha})$ satisfies the condition of Transfinite Dependent Choice, which states that every $P^\beta(\beta\leq\alpha)$ is a regular encoding that correctly computes the power set of $\beta$, and for every $\gamma<\beta\leq\alpha$, we have $P^\beta\cap\gamma=P^\gamma\cap\gamma$.
    Only one direction is non-trivial: suppose that for every $\beta<\alpha$ we have $\varphi(\beta,P^\beta,\langle P^\gamma\rangle_{\gamma<\beta})$.
    Our goal is to construct a $P^{\alpha}$ such that it is a regular encoding that correctly computes the power set of $\alpha$, and for every $\beta<\alpha$ we have $P^\alpha\cap\beta=P^{\beta}\cap\beta$; if $\alpha$ is a limit ordinal, this means that we have $P^{\alpha}\cap\alpha$ a fixed uniformly bounded hyperclass; if $\alpha=\alpha'+1$ is a successor ordinal, we just need to fix $P^{\alpha}\cap \alpha'$, which is still a uniformly bounded hyperclass.
    In both cases, uniform encodability concludes that a single preset represents the hyperclass we want, and another preset can represent the true power set of $\alpha$.
    Any regular encoding that encodes both presets can serve as the desired $P^{\alpha}$.

    Now, Transfinite Dependent Choice gives us a class $P$ such that every $P^\alpha$ is a regular encoding that correctly computes the power set of $\alpha$, and for every $\beta<\alpha$ we have $P^\beta\cap\beta=P^\alpha\cap\beta$.
    We will prove that the union of all $P^\alpha\cap\alpha$, denoted as $E$, is a regular encoding that encodes all bounded classes.
    Firstly, we will show that every bounded class is encoded as a ordinal: since every $P(\alpha)$ could be viewed as a bounded class, it has a cardinality which is the least ordinal that carries a bijection to some preset for that bounded class; now suppose for every $\beta<\alpha$ we have all subsets of $\beta$ encoded in $E$.
    Consider the cardinal $(|P(\alpha)|\backslash\operatorname{sup}\{|P_\beta|\mid \beta<\alpha\})^+$, it is easy to check that $E$ has the regularity, so at that step one must have been encoded all subsets of $\alpha$, which leads to our conclusion.
    Now other axioms are easy to check, and we only need to prove that replacement still holds.
    Suppose $\operatorname{Sat}(\ulcorner\varphi\urcorner,\alpha,E,P)$, where $\varphi(\alpha,E,P)$ claiming that for some formula $\psi_0$ with class parameter $P$, the set universe encoded by $E$ satisfies that $\forall \gamma<\alpha\exists!\delta(\psi(\gamma,\delta,P))$.
    Then by the definition of truth predicate, we have that for every $\gamma<\alpha$ there exists only one ordinal $\delta$ such that $\operatorname{Sat}(\ulcorner\psi\urcorner,\gamma,\delta,E,P)$, with $\psi(\gamma,\delta,E,P)$ stating that the set universe encoded by $E$ satisfies $\psi_0(\gamma,\delta,P)$.
    So by the replacement schema in $\mathsf{OA}$, there exists some $\beta$ such that every such $\delta$ is below $\beta$; by comprehension, a bounded class consists of all such $\delta$s exists.
    Now notice that every bounded class is encoded in $E$, this special bounded class gives the desired set in the set universe.
\end{proof}

It is easy to obtain $\mathsf{MU}+\mathsf{TDC}+\mathsf{UE}$ from $\mathsf{OA}+$ ``there exists a maximal regular encoding''.
If it exists, the universe of sets having a set-like global well-ordering implies that every class carries an injection into the class of all ordinals.
So our second order ordinal theory actually quantifies over all set-classes, and we can talk about the true second order set theory in this context.

\subsection{Type-B Universe: $\mathsf{MC}_B$ and $\mathsf{MC}_B^+$}

However, another intuition for maximality rejects any global well-ordering because that means any regular encoding can't enumerate all sets or that there are far more sets than ordinals.
This will be called relative maximality: we want to show that sets are far more than ordinals, instead of claiming that sets are as much as possible directly.
An axiom describing this intuition is as follows.

\textbf{$\mathsf{PS}$, Plenitude of Sets.} For every regular encoding $E$, there is a bounded class $e$ such that $e$ is not coded in $E$.

In fact, a naturally strengthened version of this axiom could be stated as follows.

\textbf{$\mathsf{US}$, Unlimited Size.} For every class $C$, if for every ordinal $\alpha$ we have $C_\alpha$ as a bounded class, then there is a bounded class not equal to any $C_\alpha$.

Unlimited Size, exactly the negation of the Limitation of Size, is a radical strengthening of the Plenitude of Sets.

Let us investigate how this intuition shapes our universe of sets.
We may want $\mathsf{ZFC}^-$ at least, but a strong enough choice axiom is a must for the collection schema.
Here are the related technical results.
It is natural to obtain a bounded version of the dependent choice principle and the class choice principle.

\textbf{$\mathsf{DC}$, Dependent Choice.} Fix an ordinal $\theta$. Suppose that, for every ordinal $\alpha<\theta$ and an $\alpha$-sequence of bounded classes $\langle p^\beta\rangle_{\beta<\alpha}$(which can be treated as a single bounded class with $p^\beta$ being the section of it at $\beta$), there exists a bounded class $p^\alpha$ such that $\varphi(\alpha,p^\alpha,\langle p^\beta\rangle_{\beta<\alpha})$ if and only if, for every $\beta<\alpha$, we have $\varphi(\beta,p^\beta,\langle p^\gamma\rangle_{\gamma<\beta})$.
Then, there exists a bounded class $p$ such that for every $\alpha<\theta$ we have $\varphi(\alpha,p_\alpha,\langle p_\beta\rangle_{\beta<\alpha})$.

\textbf{$\mathsf{AC}$, Axiom of Choice.} For any ordinal $\theta$, if for every $\alpha<\theta$ there exists some bounded class $c^\alpha$ such that $\psi(\alpha,c^\alpha)$, then there exists a bounded class $c$ such that $\psi(\alpha,c_\alpha)$ holds for every $\alpha<\theta$.

\begin{theorem}
    $\mathsf{OA}\vdash\mathsf{DC}\to\mathsf{AC}$.
\end{theorem}
\begin{proof}
    Just the same as the proof of $\mathsf{CC}$ from $\mathsf{TDC}$.
\end{proof}

\begin{theorem}
    Assume $\mathsf{OA}+\mathsf{AC}$; we have at least first-order set theory $\mathsf{ZFC}^-$ naturally interpreted in it.
    If we have $\mathsf{DC}$ in $\mathsf{OA}$, we can deduce $\mathsf{DC}_{<\mathbf{Ord}}$ in the related set theory.
\end{theorem}
\begin{proof}
    Let us show how to prove the Collection Principle in set theory with the $\mathsf{AC}$ above.
    Suppose that (after interpreting correctly) for every $\alpha<\theta$ and bounded class $x$, there exists at least one bounded class $p^\alpha$ such that $\varphi(x_\alpha,p^\alpha,q)$ holds.
    From $\mathsf{AC}$ we obtain another bounded class $y$, such that $\varphi(x_\alpha,y_\alpha,q)$ holds for every $\alpha<\theta$.
    This is enough to decode a set for the collection principle.
    The proof of $\mathsf{DC}_{<\mathbf{Ord}}$ from $\mathsf{DC}$ is alike.
\end{proof}

Unlike the theory $\mathsf{MC}_A$, $\mathsf{US}$ rejects any global well ordering on the universe of sets.
That is because such a global well ordering can be viewed as a class that enumerates every bounded class as its section, but $\mathsf{US}$ has just claimed that it could not be global.
So if we accept $\mathsf{US}$, we must accept that we can only derive a first order set theory from the ordinal theory.
Also, there is a least ordinal such that every ordinal below it has an encoding for its power set, but the ``power set'' of the ordinal itself is an unencodable hyperclass.
It is easy to derive that this ordinal is actually a cardinal, and we will call it the \textbf{largeness cardinal}.
One may be concerned that $\mathsf{US}$ leads to some inconsistent results; we will exclude the possibility of that case.

\begin{theorem}
    If $\mathsf{ZFC}+$an inaccessible cardinal is consistent, so is the theory $\mathsf{OA}+\mathsf{MU}+\mathsf{TDC}+\mathsf{US}$; we may denote this theory as $\mathsf{MC}_B$.
    $\mathsf{OA}^++\mathsf{MU}+\mathsf{TDC}+\mathsf{US}$ is also consistent under the hypothesis; this theory will be called $\mathsf{MC}_B^+$.
\end{theorem}
\begin{proof}
    Start from $\mathbf{L}$ with $\kappa$, the least inaccessible cardinal there.
    We consider $\operatorname{Add}(\omega,\kappa^+)=\{p\mid \text{$p$ is a partial function from $\omega\times\kappa^+$ to $2$ with }|\operatorname{dom}(p)|<\omega\}$.
    In $\mathbf{L}[G]$, $\kappa$ is still the least weakly inaccessible cardinal(because Cohen forcing is c.c.c.), and for every $\lambda<\kappa$, we have $2^\lambda=\kappa^+$($\kappa^+=(2^\omega)^{\mathbf{L}[G]}\leq(2^\lambda)^{\mathbf{L}[G]}\leq((\kappa^+)^\lambda)^{\mathbf{L}}=\kappa^+$), which means $\kappa$ is not (strongly) inaccessible.
    We will show that $(\kappa,P(\kappa);<,+,\times,\in)$ is a model of $\mathsf{MC}_B^+$.

    Most of these axioms can be checked directly; we will show that $\mathsf{MU}$ can be satisfied.
    For any $x\subseteq\alpha<\kappa$, we may notice that the set universe interpreted in the structure is actually $(H_\kappa,\in)$; since $x\in \mathbf{L}[x]$, if we could prove that $\kappa$ is inaccessible in $\mathbf{L}[x]$, $V_\kappa^{\mathbf{L}[x]}=L_\kappa[x]\ni x$ will be a subset of $H_\kappa$, and the canonical well-ordering of $\mathbf{L}[x]$ makes it a natural $\kappa$-sized model of $\mathsf{ZFGC}+<\text{ respects }\operatorname{rank}$.
    However, we all know that $\mathsf{GCH}$ holds in $\mathbf{L}[x]$ above $(|x|^+)^{\mathbf{L}[G]}$; since $\kappa$ is a limit cardinal in $\mathbf{L}[G]$, $|x|^{\mathbf{L[G]}}<\kappa$ implies $(|x|^+)^{\mathbf{L}[G]}<\kappa$, so $\kappa$ is strongly inaccessible in $\mathbf{L}[x]$, and the proof is done.

    $\mathsf{US}$ comes from an easy cardinality argument: every bounded class in this structure has cardinality $<\kappa$ and is $\in H_\kappa$, so a class of bounded classes is a $\kappa$-length sequence of sets in $H_\kappa$; however, $|H_\kappa|=\kappa^+$, so there are always some bounded classes that are not enumerated in the sequence.
\end{proof}
 
It is easy to observe that the largeness cardinal in the model constructed above is $\omega$.
If we want to change it, we could just add $\kappa^+$-many Cohen subsets to another cardinal; in fact, we can control the continuum function below that largeness cardinal by taking products with some other small forcing notions(relative to $\kappa$) if we wish.

Also, a model of $\mathsf{MC}_B+\neg$ ``every cardinal has a successor'' can be constructed by collapsing every cardinal below $\kappa$ to $\omega$ and adding $\kappa^+$-many Cohen reals, which makes $\kappa=\omega_1^{\mathbf{L}[G]}$ inaccessible to reals, and the proofs of $\mathsf{MU}$ and $\mathsf{US}$ are alike.

\subsection{The Power Set Dichotomy}

There is no standard model after combining $\mathsf{PS}$ with $\mathsf{UE}$, and adding $\mathsf{MU}$ and $\mathsf{TDC}$ leads to an unacceptable inconsistency.
This dichotomy shows that there are only two coherent ways to proceed: either one accepts that the universe of all sets satisfies $\mathsf{ZFGC}^2$, so that the power set operation is globally well-defined, or one accepts that there must be ordinals at which the power set operation fails to deliver the ``correct'' cardinal.
We refer to this as the \emph{Power Set Dichotomy}.
How can we make a choice in this dichotomy?

First, let us examine whether the arguments in support of the Power Set Axiom found in the original writings of Gödel are sufficient to incline us toward $\mathsf{MC}_A$.
In G\"odel’s later reflections, as reported by Wang\cite{Wang}, the formation of power set is associated with a ``jump'' to a higher level of cardinalities:

\begin{quote}
    8.2.13 The power-set operation involves a jump. In this second jump we consider not only the members of a set as given but also the process of selecting members from the set. Taking all possible ways of leaving out members of the set is a kind of "method" for producing all its subsets. We then feel that we can overview the collection of all these subsets as well. We idealize, for instance, the integers or the finite sets (a) to the possibility of an infinite totality, and (b) with omissions. So we get a concretely intuitive idea and then one goes on. There is no doubt in the mind that this idealization-to any extent whatsoever-is at the bottom of classical mathematics. (Compare 7.1.18 and 7.1.19.)
\end{quote}

This ``jump'' is most naturally understood as the passage from a set $x$ to its power set $P(x)$, which produces a strictly richer domain.
Unlike Cantor’s ordinal progression, this operation is not merely structural, but depends on the totality of subsets of $x$, and hence implicitly on the surrounding universe.

Cohen’s forcing method reveals a crucial feature of this situation: the size of the ``jump'' cannot be fixed independently of the ambient set-theoretic universe.
In particular, the enumeration of $P(x)$ is not determined locally, but varies with global properties of the universe.
This reflects the fact that the statement $x = P(y)$ involves a quantification over all subsets of $y$, and is therefore essentially impredicative.

From this perspective, one may ask whether such a jump genuinely requires the Power Set Axiom.
There is an alternative candidate: the Hartogs operation, which assigns to each set $x$ the least ordinal not embeddable into $x$, thereby producing a canonical successor cardinal without invoking the full power set.

From a G\"odelian point of view, one might attempt to argue that this distinction collapses under the assumption of $\mathsf{GCH}$.
If $\mathsf{GCH}$ holds, then the Hartogs number and the power set of $x$ are equinumerous, so that both constructions determine the same ``next stage''.
In this sense, the passage to a larger domain would appear to be uniquely determined.

This line of reasoning can be seen as a local form of the Limitation of Size principle: whenever one attempts to pass beyond a given set, all sufficiently canonical constructions yield the same result.
However, this generalization is highly problematic.
Since Limitation of Size is not first-order expressible, there is no reason to expect that the cardinals reflecting global features of the universe form a well-behaved class.
Indeed, many global properties of $V$ systematically fail to reflect to initial segments, and there is no clear argument that Limitation of Size should be exempt from such failures.

We therefore conclude that the Cantor-G\"odel style justification of the Power Set Axiom (and of $\mathsf{GCH}$) does not establish the full strength of the axiom.
At best, it supports a weaker principle: that every cardinal has a successor.
This is also consistent with our earlier technical results: the natural interpretation of the first‑order set‑theoretic fragment shared by $\mathsf{MC}_A$ and $\mathsf{MC}_B^+$ is precisely $\mathsf{ZFC}^- +$ “every cardinal has a successor”.

At this point, two further considerations support our preference for $\mathsf{MC}_B^+$ over $\mathsf{MC}_B$: one extrinsic and one intrinsic.

From an extrinsic perspective, it is natural to require that the universe of sets be \emph{rigid}, in the sense that it admits no nontrivial elementary self-embeddings, as is supported by almost all candidate characterizations of the true set-theoretic universe.
Indeed, it is not difficult to see that $\mathsf{MC}_A$ proves that there is no nontrivial elementary embedding from $V$ to $V$. 
Similarly, $\mathsf{MC}_B^+$ rules out any such embedding that can be coded as a class in the corresponding ordinal theory (cf.\ \cite[Thm.\ 5.2]{Matthews}). 
By contrast, a straightforward translation of \cite[Thm.\ 2.2]{Matthews} shows that, assuming the consistency of $\mathsf{ZFC}+I_1$, the system $\mathsf{MC}_B$ admits a nontrivial elementary embedding of $V$ into itself (albeit non-cofinal) which can be coded as a class. 
This indicates that the universe described by $\mathsf{MC}_B$ is not rigid, a feature that is difficult to reconcile with the idea of a well-determined totality of sets.

From an intrinsic perspective, our analysis has aimed to show that the notion of cardinal number is subordinate to that of ordinal number, and that cardinal arithmetic should accordingly be understood as derivative from ordinal arithmetic. 
This view is already explicit in Cantor’s original conception. 
On the one hand, Cantor repeatedly emphasizes that there is no largest transfinite number; every number class gives rise to a strictly larger one, and the sequence of transfinite numbers is in principle unbounded (cf.\ \cite{Cantor1883}, esp.\ the discussion of the unending progression of number classes). 
On the other hand, cardinal numbers are introduced only after the theory of order-types and are obtained by a further abstraction from well-ordered sets; in this sense, they depend conceptually on ordinal structure (cf.\ \cite{Cantor1895}). 
Thus, the idea that cardinals are secondary to ordinals is not an artifact of later set theory, but reflects a fundamental feature of Cantor’s own framework.

From this standpoint, it is natural to regard $\mathsf{OA}^+$ as providing the complete theory of ordinal arithmetic, with cardinal structure emerging only as a derived notion. 
This reinforces the idea that the passage to higher domains should be governed by ordinal principles rather than by independent cardinal postulates, and hence supports the preference for $\mathsf{MC}_B^+$ as the more faithful realization of the Cantorian picture.

At this point, let us return to the main line of argument and ask whether there is any alternative way to resolve the dichotomy.

One possible maneuver would be to appeal to a form of potentialism, according to which the choice between the two options may legitimately be suspended. 
This idea rests on the following observation: although $\mathbf{Ord}$ is quantified over in our second-order theory of ordinals, it is not itself admitted as a mathematical object (that is, as a set). 
Accordingly, it is natural to regard the first-order set theory induced by our second-order ordinal theory as describing only an initial segment of the true universe of sets. 
On this view, we are characterizing such initial segments via principles akin to reflection, and then extrapolating from these local descriptions to a global picture of the universe.

From this perspective, it is not difficult to see—by a standard forcing argument—that, assuming the existence of two inaccessible cardinals, one may have cardinals $\kappa < \lambda$ such that $(\kappa, P(\kappa)) \vDash \mathsf{MC}_A$ while $(\lambda, P(\lambda)) \vDash \mathsf{MC}_B^+$; and conversely, one may also arrange that $(\kappa, P(\kappa)) \vDash \mathsf{MC}_B^+$ while $(\lambda, P(\lambda)) \vDash \mathsf{MC}_A$. 
This suggests that Type A and Type B universes can transform into one another along further ordinal extensions lying beyond the current scope of our theory. 
A potentialist may therefore conclude that no definitive choice between the two is possible.

However, we shall argue that this line of reasoning is ultimately untenable.

For convenience, let us refer to such extensions beyond the present theoretical scope as \emph{hyper-iterations} of the universe. 
Suppose first that a given universe cannot be hyper-iterated into a larger universe of the other type. 
In that case, we already have decisive grounds for regarding its type as correct, and hence as stable under all further extensions.

Assume, then, that this does not occur. 
Consider the behavior of largeness cardinals in a Type B universe under hyper-iteration. 
This is a local question: once we pass to the next Type B universe, the status of such cardinals is determined within that universe and will not subsequently change. 
Hence, the question has a definite truth value and cannot be indefinitely suspended by a potentialist stance.

If it turns out that, for every Type B universe, its largeness cardinals fail to remain such after hyper-iteration, then each of them must be dominated by larger cardinals in the extended universe. 
In that case, their power sets must in fact be sets (i.e., bounded classes), and the original Type B universe failed to recognize this only because the ordinal extension was not sufficiently developed. 
But if this situation occurs universally, it follows that the power set of every largeness cardinal is genuinely a set, and hence that $\mathsf{UE}$ is justified; this, in turn, implies that $\mathsf{MC}_A$ is the correct theory.

On the other hand, if this situation does not eventually occur, then there must exist a maximal largeness cardinal such that all subsequent Type B universes take it to be their largeness cardinal. 
This entails that no such universe can be further hyper-iterated into a Type A universe, and hence that $\mathsf{MC}_B^+$ is ultimately correct.

Therefore, even when hyper-iteration is taken into account, the dichotomy cannot be avoided: one must still make a definite choice between Type A and Type B universes in order to determine the correct theory.

Moreover, this potentialist perspective also provides us with a fairly good justification for $\mathsf{TDC}$.
Since we must add the two principles $\mathsf{MU}$ and $\mathsf{TDC}$ to the theory in order to degrade the contradiction from the standard semantics of second‑order logic to the syntax of second‑order logic, it is certainly necessary to carry out some justification for both.
$\mathsf{MU}$ requires little attention, because we have already established the status of $\mathsf{SC}^{\text{reg}}$; in a second‑order context, $\mathsf{MU}$ is simply a reaffirmation of the correctness of $\mathsf{SC}^{\text{reg}}$.

As for $\mathsf{TDC}$, we might as well invoke Cantor’s principle of Free Conceptual Construction once again.
First, according to our potentialist perspective, the class that $\mathsf{TDC}$ seeks to construct should also be a legitimate mathematical object; thus, the existence of this mathematical object is transformed into the question of whether it can be coherently embedded into the previous context—and this is precisely what the first half of the axiom guarantees: it exists coherently in every local context that could be spoken of.
Therefore, by the maximality that the mathematical universe ought to possess, this object should also exist as a whole, which just justifies $\mathsf{TDC}$.

\subsection{A Glance at Large Cardinals}

As G\"odel's program may suggest, if we can't decide which picture of the set universe is correct, we should consider if some large cardinal hypothesis can help us with it.
Immediately we can have, if strongly inaccessible cardinal is cofinal in our universe, then the power set axiom holds everywhere.
However, the definition of strongly inaccessibles is just linked to the power set axiom, and a more neutral way for claiming this is clearly to have that strongly inaccessibles are cofinal below the largeness cardinal, which is innocent for deciding the power set axiom worldwidely.
Another way is to have that weakly inaccessible cardinal is cofinal in our universe, but it is also innocent.

We may look at larger cardinals.
Notice that, taking $\kappa\to(\kappa)^2_2$ as the definition of weakly compact cardinals, we can have an axiom claiming that $\mathbf{Ord}$ is weakly compact: for every class $F$, if it can be decoded as a binary coloring function $\mathbf{Ord}^2\to 2$, then there is a cofinal class $H$ that is homogeneous according to this function.
\begin{theorem}
    $\mathbf{Ord}$ is weakly compact implies $\mathsf{UE}$ in $\mathsf{OA}^++\mathsf{TDC}$.
\end{theorem}
\begin{proof}
    Suppose that some uniformly bounded hyperclass cannot be encoded as a single bounded class.
    Therefore, for some fixed cardinal $\kappa$, for every enumeration of sets bounded by $\kappa$ of ordinal-length, there is another set, also bounded by $\kappa$, not enumerated by the previous one.
    From $\mathsf{TDC}$, we can get a class $C$, with every slice of it being a set bounded by $\kappa$, pairwise distinct.
    Consider the following lexicographical order $<_{lex}$ on all these slices: put $\alpha<_{lex}\beta$ if and only if $\min(C^{\alpha}\backslash C^\beta)<\min(C^\beta\backslash C^\alpha)$.
    Now set $F(\alpha,\beta)=0$ if and only if $\alpha<\beta$ and $\alpha<_{lex}\beta$ or vise versa.
    $\mathbf{Ord}$ is weakly compact means that there is a homogeneous cofinal class $D$, or equivalently, $<_{lex}$ has an $\mathbf{Ord}-$length increasing or decreasing sequence.
    We will derive a contradiction from $<_{lex}$ having an $\kappa^+-$length increasing sequence, and the proof for the other case is just alike.

    Assume without loss of generality that $p^\alpha(\alpha<\kappa^+)$ is the enumeration of this $\kappa^+-$length increasing sequence; we consider $p^\alpha$ as a function $f^\alpha:\kappa\to 2$, with $\theta\in p^\alpha\leftrightarrow f(\theta)=1$ for every $\theta<\kappa$.
    Consider $x_\theta=\{f^\alpha\upharpoonright\theta\mid\alpha<\kappa^+\}$, the encodability of this hyperclass immediately follows from replacement; we can consider the minimum cardinal that is sufficient for encoding it, which is actually the cardinality of this family of sets.
    If $|x_\theta|\leq\kappa$ for all $\theta$, $\{f^\alpha\mid\alpha<\kappa^+\}$ must have a cardinality less than or equal to $\kappa$, a contradiction; therefore, there exists some minimum $\theta$ ensuring $|x_\theta|=\kappa^+$.
    Now we can assume again without loss of generality that $f^\alpha\upharpoonright\theta$ are pairwise distinct.
    If $\alpha<\kappa^+$, $p^\alpha<_{lex}p^{\alpha+1}$, so there exists only one $\beta_\alpha$ such that $f^\alpha\upharpoonright\beta_\alpha=f^{\alpha+1}\upharpoonright\beta_\alpha$ and $f^\alpha(\beta_\alpha)=1>0=f^{\alpha+1}(\beta_\alpha)$.
    Obviously, $\beta_\alpha$ are all less than $\theta$, so $\beta:\kappa^+\to\theta$ must give us some cofinal subset $x$ of $\kappa^+$ on which $\beta$ takes the same value $\theta'$, and now $\theta$ is not the least one having our desired property; a contradiction.
\end{proof}

Similar theorems can't be done if we take the assumptions that $\mathbf{Ord}$ is weakly Mahlo, or any strengthening by iterating the property of weakly inaccessible from below.
Just as we define $\kappa$ to be weakly Mahlo if weakly inaccessibles are cofinal below $\kappa$, we can define $1$-weakly Mahlo, $\alpha$-weakly Mahlo and, after all, weakly totally Mahlo, which states that for every $\alpha<\kappa$, $\alpha$-weakly Mahlo cardinals are cofinal below $\kappa$.
\begin{theorem}
    $\mathbf{Ord}$ is weakly totally Mahlo does not imply $\mathbf{UE}$ in $\mathsf{OA}^++\mathsf{TDC}$.
\end{theorem}
\begin{proof}
    By standard forcing techniques, we could create a weakly totally Mahlo cardinal that is weakly inaccessible but not strongly inaccessible.
    The rest of the proof is just translating $V_\kappa$ into a $\mathsf{OA}^++\mathsf{TDC}-$structure as we've done several times before.
\end{proof}

However, how can we convince that weakly compact is a suitable large cardinal assumption, especially for $\mathbf{Ord}$? The preceding argument shows that the weak compactness of $\mathbf{Ord}$ entails $\mathsf{UE}$ and thus forces the universe into the Type‑A picture. But this very fact raises a suspicion: the assumption of $\mathbf{Ord}$ being weakly compact is not neutral with respect to the Power Set Dichotomy. Indeed, the assertion that every binary class‑function on $\mathbf{Ord}$ admits a homogeneous cofinal class is a strong homogeneity requirement on the universe. It says that the ordinals, together with all classes, cannot be arranged in a highly chaotic manner; there is always a thread of uniformity. By contrast, the intuition behind the Type‑B universe is precisely anti‑homogeneous: we want the universe of sets to be so rich and irregular that no ordinal‑length tower can exhaust all bounded classes. The chaotic abundance of subsets that characterizes the Type‑B picture is diametrically opposed to the homogeneous orderliness demanded by weak compactness. Hence, postulating $\mathbf{Ord}$ to be weakly compact is tantamount to ruling out the Type‑B alternative from the outset; it is not a hypothesis that can be weighed independently of the dichotomy.

If we seek a large cardinal principle that does not already encode a commitment to the power set axiom, we might instead consider a weaker property that is equiconsistent with weak compactness but avoids the homogeneity consequences. A natural candidate is the tree property, parallel to the way we earlier replaced strong inaccessibility by weak inaccessibility. It would allow the chaotic structure required by the Type‑B universe, while still expressing a form of compactness that might be acceptable from a maximality standpoint. Consequently, even the most natural large cardinal postulates for $\mathbf{Ord}$ do not force a resolution of the Power Set Dichotomy; the choice between Type‑A and Type‑B universes remains open on the basis of large cardinal considerations alone.

\section{Conclusion}

Let us summarize the current situation.
Based on Cantor's Thesis on Ordinals as we have read from Cantor's writings, we have arrived at two conclusions.
First, we can select an enumeration of bounded sets of ordinals to obtain an appropriate number (i.e., as many as the ordinals) of bounded sets of ordinals, and thereby provide a complete first‑order characterization of the theory of ordinals, namely the $\mathsf{SC}^{\text{reg}}$ system.
Second, in order to study full first‑order set theory, we have examined two possible second‑order theories of ordinals: one stating that sets are as many as ordinals, yielding $\mathsf{MC}_A$; the other stating that sets are far more numerous than ordinals, yielding $\mathsf{MC}_B^+$.
These two possible theories are incompatible. By deriving this incompatibility from the global level down to the local level in our theoretical development, we arrive at the Power Set Dichotomy: the only issue on which we must take a stance is whether the power set of every ordinal is a set.
Through philosophical arguments, we have shown that neither side of this dichotomy has yet received a satisfactory philosophical justification, but we have also shown that this dichotomy cannot be left unresolved.
Therefore, perhaps we can summarize the work we have accomplished in the following conclusion: $\mathsf{ZFC}^-+$ ``every cardinal has a successor'' is a fragment of true first‑order set theory, and it is necessary to choose between two possible conceptions of the universe that diverge in their attitude toward the Power Set Axiom.

\appendix

\section{On definablilty of G\"odel pairing function}

We will prove the following technical theorems here.
\begin{theorem}
    $\operatorname{Pair}$ is definable using addition and multiplication.
\end{theorem}

\begin{proof}
We first derive the definition of the special case $f(\alpha)=\operatorname{Pair}(\alpha,\alpha)$.

It is straightforward to see that $f(\alpha)$ is determined by the recursive relations
\(f(0)=0\),
\(f(\alpha+1)=f(\alpha)+1+2\alpha\),
and
\(f(\alpha)=\bigcup_{\beta<\alpha}f(\beta)\) for limit ordinals.

\(0\) is obviously definable; some other trivially definable predicates are:
\[
\begin{aligned}
\operatorname{isSucc}(\alpha) &:= \exists\beta<\alpha\forall\gamma<\alpha(\gamma\leq\beta),\,\,
\operatorname{isLim}(\alpha) := \neg\operatorname{isSucc}(\alpha),\\
\operatorname{isNat}(\alpha) &:= \forall\beta<\alpha\bigl(\operatorname{isSucc}(\beta)\lor\beta=0\bigr)\wedge\operatorname{isSucc}(\alpha).
\end{aligned}
\]
Thus \(\omega\) is also a definable constant.

We further define
\[
\begin{aligned}
\operatorname{isAddIndecomposable}(\alpha) &:= \forall\beta,\gamma<\alpha(\beta+\gamma<\alpha),\\
\operatorname{isMulIndecomposable}(\alpha) &:= \forall\beta,\gamma<\alpha(\beta\times\gamma<\alpha).
\end{aligned}
\]

By well-known results (see, e.g., Jech), an ordinal \(\alpha\) is additively indecomposable if and only if it is of the form \(\omega^\theta\), and it is multiplicatively indecomposable if and only if this \(\theta\) is additively indecomposable.

We can define \(\omega^\omega\) as the smallest multiplicatively indecomposable ordinal greater than \(\omega\). Then all \(\omega^n\) (where \(n\) is a natural number) can be defined as all additively indecomposable ordinals less than \(\omega^\omega\).

Furthermore, we may define ordinals of the form \(\omega^{\omega^\theta\times n}\), which we call \textbf{semi-multiplicatively indecomposable ordinals}. An ordinal \(\alpha\) is such an ordinal if and only if it is additively indecomposable and, letting \(\alpha'\) be the largest multiplicatively indecomposable ordinal not exceeding \(\alpha\), for any additively indecomposable ordinal \(\gamma<\alpha'\), if there exists an ordinal \(\beta\) such that \(\beta\times\gamma=\alpha\), then \(\gamma=1\).

This follows from the fact that every additively indecomposable ordinal has the form \(\omega^{\omega^{\theta_1}\times n_1+\xi}\) with \(\xi<\omega^{\theta_1}\). The largest multiplicatively indecomposable ordinal not exceeding it is then \(\omega^{\omega^{\theta_1}}\). Suppose for contradiction that there exists an additively indecomposable ordinal \(\gamma=\omega^{\gamma'}\) and an ordinal \(\beta=\omega^{\beta'}\) with \(0<\gamma'<\omega^{\beta_1}\) such that \(\beta\times\gamma=\alpha\). Then \(\alpha=\omega^{\beta'+\gamma'}\), and \(\gamma'>0\) imply \(\xi>0\). The existence of \(\beta\) given \(\gamma\) is straightforward.

We now prove that \(f(\alpha)\) is defined by the following case distinction:
\begin{enumerate}
    \item If \(\operatorname{isNat}(\alpha)\), then \(f(\alpha)=\alpha\times\alpha\).

    \item Otherwise, let \(\theta\) be the largest additively indecomposable ordinal not exceeding \(\alpha\). It is easy to verify that there exists a natural number \(n\) and an ordinal \(\beta<\theta\) such that \(\alpha=\theta\times n+\beta\), with \(\theta\geq\omega\) and \(n\geq 1\).
        \begin{enumerate}
            \item If \(\beta=0\):
                \begin{itemize}
                    \item If \(n>1\), then \(f(\alpha)=\theta\times\theta\times(n-1)\).
                    \item If \(n=1\), we further case on \(\alpha=\theta\):
                        \begin{itemize}
                            \item If \(\theta\) is multiplicatively indecomposable, then \(f(\theta)=\theta\).
                            \item If \(\theta\) is not multiplicatively indecomposable and there exists a largest additively indecomposable ordinal \(\theta'<\theta\), then \(f(\alpha)=\theta'\times\theta\).
                        \end{itemize}
                \end{itemize}
                If \(\theta\) is not yet classified, then the largest multiplicatively indecomposable ordinal \(\theta_1\) less than \(\theta\) and the largest semi-multiplicatively indecomposable ordinal \(\theta_2\) not exceeding \(\theta\) are distinct. Then there exists a semi-multiplicatively indecomposable ordinal \(\theta_3\) such that \(\theta_1\times\theta_3=\theta_2\).
                \begin{enumerate}
                    \item If \(\theta_2\neq\theta\), then \(f(\alpha)=\theta_2\times\theta\).
                    \item Otherwise, we must have \(\theta_3\geq\theta_1\), and then \(f(\alpha)=\theta_3\times\theta\).
                \end{enumerate}

            \item If \(\beta>0\), decompose \(\beta=\gamma+\delta\) where \(\gamma\) is a limit ordinal and \(\delta\) is a natural number.
                \begin{itemize}
                    \item If \(\delta=0\), then \(f(\alpha)=f(\theta\times n)+\theta\times n\times\gamma\).
                    \item If \(\delta>0\), then
                    \[
                    f(\alpha)=f(\theta\times n)+\theta\times n\times\gamma+(\theta\times n+\gamma)\times 2\delta+(\delta-1).
                    \]
                \end{itemize}
        \end{enumerate}
\end{enumerate}

We proceed by induction on \(\alpha\) to show that this case distinction indeed defines the desired function \(f\) satisfying the given recursive relations. The cases where \(\alpha\) is a natural number or \(\omega\) are trivial.

Write \(\alpha\) in Cantor normal form:
\[
\alpha=\omega^{\alpha_1}\times k_1+\dots+\omega^{\alpha_m}\times k_m,\quad \alpha_1>0.
\]
Then \(\theta=\omega^{\alpha_1}\) and \(n=k_1\).
\begin{itemize}
    \item If \(\alpha_m=0\), then \(\gamma=\omega^{\alpha_2}\times k_2+\dots+\omega^{\alpha_{m-1}}\times k_{m-1}\) and \(\delta=k_m\).
    \item If \(\alpha_m>0\), then \(\gamma=\omega^{\alpha_2}\times k_2+\dots+\omega^{\alpha_m}\times k_m\) and \(\delta=0\).
\end{itemize}

We argue by the corresponding cases:
\begin{enumerate}
    \item When \(m=1\), this corresponds to \(\beta=0\). We distinguish four subcases based on \(\alpha_1\) and \(n\).
        \begin{enumerate}
            \item If \(\alpha_1=\alpha_1'+1\) is a successor ordinal and \(n>1\):
                \(\alpha=\omega^{\alpha_1}\times n\) is the limit of \(\omega^{\alpha_1}\times(n-1)+\omega^{\alpha_1'}\times k\) for \(k\in\omega\). The value of \(f\) at the latter is
                \[
                f(\omega^{\alpha_1}\times(n-1))+\omega^{\alpha_1}\times(n-1)\times\omega^{\alpha_1'}\times k,
                \]
                so the value at the former is its limit:
                \[
                f(\omega^{\alpha_1}\times(n-1))+\omega^{2\alpha_1}.
                \]
                \begin{itemize}
                    \item If \(n=2\), then
                    \[
                    f(\omega^{\alpha_1}\times(n-1))=f(\omega^{\alpha_1'+1})
                    =\omega^{\alpha_1'}\times\omega^{\alpha_1'+1}
                    =\omega^{\alpha_1'\cdot 2+1},
                    \]
                    so the total limit is \(\omega^{2\alpha_1}=\omega^{\alpha_1}\times\omega^{\alpha_1}\times(n-1)\).
                    \item If \(n>2\), then
                    \[
                    f(\omega^{\alpha_1}\times(n-1))=\omega^{\alpha_1}\times\omega^{\alpha_1}\times(n-2),
                    \]
                    and their sum is \(\omega^{\alpha_1}\times\omega^{\alpha_1}\times(n-1)\).
                \end{itemize}

            \item If \(\alpha_1\) is a limit ordinal and \(n>1\):
                Let \(\alpha_\xi\) be any cofinal sequence in \(\alpha_1\). Then \(\omega^{\alpha_1}\times n\) is the limit of \(\omega^{\alpha_1}\times(n-1)+\omega^{\alpha_\xi}\). The value of \(f\) at the latter is
                \[
                f(\omega^{\alpha_1}\times(n-1))+\omega^{\alpha_\xi},
                \]
                so the value at the former is its limit:
                \[
                f(\omega^{\alpha_1}\times(n-1))+\omega^{\alpha_1}.
                \]
                The rest follows similarly.

            \item If \(\alpha_1=\alpha_1'+1\) is a successor ordinal and \(n=1\):
                \(\omega^{\alpha_1}\) is the limit of \(\omega^{\alpha_1'}\times n\). The value of \(f\) at the latter is \(\omega^{\alpha_1'\cdot 2}\times(n-1)\), so the value at the former is its limit:
                \[
                \omega^{\alpha_1'\cdot 2+1}=\omega^{\alpha_1'}\times\omega^{\alpha_1}.
                \]

            \item If \(\alpha_1\) is a limit ordinal that is not additively indecomposable and \(n=1\):
                Decompose
                \[
                \alpha_1=\omega^{\alpha_1'}\times n'+\zeta+m,
                \]
                where \(\zeta<\omega^{\alpha_1'}\) is a limit ordinal and \(m\) is a natural number. Then
                \[
                \alpha=\omega^{\omega^{\alpha_1'}\times n'}\times\omega^\zeta\times\omega^m,
                \]
                \(\theta_1=\omega^{\omega^{\alpha_1'}}\), \(\theta_2=\omega^{\omega^{\alpha_1'}\times n'}\), \(\theta_3=\omega^{\omega^{\alpha_1'}\times(n'-1)}\).
                Let \(\theta_4=\omega^\zeta\) and \(\theta_5=\omega^m\).
                \begin{enumerate}
                    \item If \(m>0\):
                        \(\alpha\) is the limit of \(\omega^{\omega^{\alpha_1'}\times n'+\zeta+m-1}\times k\). The value of \(f\) at the latter is
                        \[
                        \omega^{\omega^{\alpha_1'}\cdot 2n'+\zeta+m-1}\times(k-1),
                        \]
                        so the value at the former is its limit:
                        \[
                        \omega^{\omega^{\alpha_1'}\cdot 2n'+\zeta+m}=\theta_2\times\theta.
                        \]

                    \item If \(m=0\) but \(\zeta>0\):
                        Let \(\zeta_\xi+m_\xi\) be a cofinal sequence in \(\zeta\) where each \(\zeta_\xi\) is a limit ordinal and each \(m_\xi>0\) is a natural number. Then \(\alpha=\theta\) is the limit of
                        \[
                        \omega^{\omega^{\alpha_1'}\times n'+\zeta_\xi+m_\xi},
                        \]
                        and the value of \(f\) at the latter is \(\theta_2\) times itself. Thus the value at the former is its limit: \(\theta_2\times\theta\).

                    \item If \(m=\zeta=0\), then \(n'\geq 2\). \(\alpha\) is the limit of
                        \[
                        \omega^{\omega^{\alpha_1'}\times(n'-1)+\xi}
                        \]
                        as \(\xi\to\omega^{\alpha_1'}\). The value of \(f\) at the latter is \(\theta_3\) times itself, so the value at the former is its limit: \(\theta_3\times\theta\).
                \end{enumerate}

            \item If \(\alpha_1\) is additively indecomposable and \(n=1\):
                Let \(\alpha_\xi+1\) be any cofinal sequence in \(\alpha_1\). Then \(\omega^{\alpha_1}\) is the limit of \(\omega^{\alpha_\xi+1}\). The value of \(f\) at the latter is
                \[
                \omega^{\alpha_\xi}\times\omega^{\alpha_\xi+1}=\omega^{\alpha_\xi\cdot 2+1},
                \]
                so the value at the former is its limit: \(\omega^\alpha\).
        \end{enumerate}

    \item When \(m>1\), this corresponds to \(\beta>0\). We handle two subcases.
        \begin{enumerate}
            \item If \(\alpha_m>0\) (i.e., \(\delta=0\)):
                Let \(\gamma_\xi+\delta_\xi\) be a cofinal sequence in \(\gamma\) where each \(\gamma_\xi\) is a limit ordinal and each \(\delta_\xi>0\) is a natural number. Then \(\alpha=\theta\times n+\gamma\) is the limit of
                \[
                \theta\times n+\gamma_\xi+\delta_\xi.
                \]
                The value of \(f\) at the latter is
                \[
                f(\theta\times n)+\theta\times n\times\gamma_\xi+(\theta\times n+\gamma_\xi)\times 2\delta_\xi+(\delta_\xi-2),
                \]
                so the value at the former is its limit:
                \[
                f(\theta\times n)+\theta\times n\times\gamma.
                \]

            \item If \(\alpha_m=0\) (i.e., \(\delta>0\)):
                We induct directly on \(\delta\), given that \(f(\theta\times n+\gamma)\) has already been computed. The conclusion is immediate.
        \end{enumerate}
\end{enumerate}

In summary, we have established the definability of \(f\). We now consider an arbitrary \(\operatorname{Pair}(\alpha,\beta)\).

If \(\alpha=\beta=\theta\), then \(\operatorname{Pair}(\alpha,\beta)=f(\theta)\) is already known. We distinguish the cases \(\alpha>\beta\) and \(\alpha<\beta\):
\begin{enumerate}
    \item If \(\alpha>\beta\), let \(\gamma\) be the least ordinal such that \(\beta+\gamma=\alpha\). Then \(\operatorname{Pair}(\alpha,\beta)\) is the unique ordinal \(\delta\) such that \(\delta+\gamma=f(\alpha)\).

    \item If \(\alpha<\beta\), let \(\gamma\) be the least ordinal such that \(\alpha+\gamma=\beta\). Then \(\operatorname{Pair}(\alpha,\beta)\) is the unique ordinal \(\delta\) such that \(\delta+(\gamma+\beta)=f(\beta)\).
\end{enumerate}

This completes the proof.
\end{proof}

\begin{theorem}
    $\mathsf{OA}$ proves that $\operatorname{Pair}$ is an increasing pairing function.
\end{theorem}
\begin{proof}
    Comprehension and well-foundedness allow us to prove formulae by transfinite induction.
    We first prove that $\forall\alpha\forall\beta\forall\gamma(\alpha+(\beta+\gamma)=(\alpha+\beta)+\gamma)$: put induction on $\gamma$, $(\alpha+\beta)+\gamma=\operatorname{ssup}\{(\alpha+\beta)+\delta\mid \delta<\gamma\}$, and $\alpha+(\beta+\gamma)=\alpha+\operatorname{ssup}\{\beta+\delta\mid\delta<\gamma\}=\operatorname{ssup}\{\alpha+\xi\mid\xi\leq\beta+\delta,\delta<\gamma\}=\operatorname{ssup}\{\alpha+\xi\mid\xi=\beta+\delta,\delta<\gamma\}=\operatorname{ssup\{\alpha+(\beta+\delta)\mid \delta<\gamma\}}$.
    Now $\alpha+(\beta+\delta)=(\alpha+\beta)+\delta$ follows from the induction hypothesis, and we have $\{(\alpha+\beta)+\delta\mid \delta<\gamma\}=\{\alpha+(\beta+\delta)\mid\delta<\gamma\}$, so they have the same $\operatorname{ssup}$, which is both $\alpha+(\beta+\gamma)$ and $(\alpha+\beta)+\gamma$.

    The same technique can be used to prove the following results.
    \begin{enumerate}
        \item $\forall\alpha,\beta,\gamma(\alpha\times(\beta+\gamma)=\alpha\times\beta+\alpha\times\gamma)$.
        \item $\alpha\times(\beta\times\gamma)=(\alpha\times\beta)\times\gamma$.
        \item If $\beta<\alpha$, then there is a unique $\gamma$ such that $\beta+\gamma=\alpha$.
        \item If $\alpha$ is a limit of limit ordinals, then $\alpha$ is itself a limit ordinal.
        \item Every ordinal $\alpha$ is equivalent to $\beta+n$, where $\beta$ is the largest limit ordinal $\leq\alpha$ and $n$ is a natural number.
        \item If $\alpha$ is a limit of additively indecomposable ordinals, then $\alpha$ is itself an additively indecomposable ordinal.
        \item If $\alpha$ is an additively indecomposable ordinal, then $\alpha\times\omega$ is also an additively indecomposable ordinal.
        \item Every ordinal $\alpha$ is equivalent to $\beta\times n+\gamma$, where $\beta$ is the largest additively indecomposable ordinal $\leq\alpha$, $n$ is a natural number, and $\gamma<\beta$.
        \item If $\beta<\alpha$ are two additively indecomposable ordinals, there is a unique additively indecomposable ordinal $\gamma$ such that $\beta\times\gamma=\alpha$.
        \item If $\alpha$ is a limit of multiplicatively indecomposable ordinals, then $\alpha$ is itself a multiplicatively indecomposable ordinal.
        \item If $\alpha$ is a semi-multiplicatively indecomposable ordinal and $\beta$ is the largest multiplicatively indecomposable ordinal $\leq\alpha$, then $\alpha\times\beta$ is also a semi-multiplicatively indecomposable ordinal, and $\beta$ is still the largest multiplicatively indecomposable ordinal $\leq\alpha\times\beta$.
        \item If $\alpha$ is a limit of semi-multiplicatively indecomposable ordinals, then $\alpha$ is itself a semi-multiplicatively indecomposable ordinal.
        \item Every additively indecomposable ordinal $\alpha$ is equivalent to $\beta\times\gamma$, where $\beta$ is the largest semi-multiplicatively indecomposable ordinal $\leq\alpha$, and $\gamma<\delta$, with $\delta$ being the largest multiplicatively indecomposable ordinal $\leq\beta$.
    \end{enumerate}
    The rest of the proof follows easily.
\end{proof}

\section{A definition of the first-order truth predicate}

We work in the monadic second‑order theory of ordinals $\mathsf{OA}$ described in Section 3.
The first‑order variables range over ordinals $\alpha,\beta,\dots$, the second‑order variables range over classes $X,Y,\dots$ of ordinals, and the language contains the symbols $<,+,\times$ on ordinals together with the membership relation $\in$ between ordinals and classes.
Fix a finite tuple of class parameters $\vec{P}=P_1,\dots,P_m$. These parameters are given once and for all and will occur as second‑order free variables in the definition of the truth predicate.
We aim to define, for every first‑order formula $\varphi$ (with first‑order free variables among $v_1,v_2,\dots$ and class parameters among $\vec{P}$), a predicate
\[
\operatorname{Sat}(\ulcorner\varphi\urcorner,x_1,\dots,x_k)
\]
expressible in the second‑order language with the same parameters $\vec{P}$, such that the Tarskian equivalence
\[
\operatorname{Sat}(\ulcorner\varphi\urcorner,x_1,\dots,x_k)\;\longleftrightarrow\;\varphi(x_1,\dots,x_k,\vec{P})
\]
is provable in $\mathsf{OA}$.
The construction uses only the ordinal arithmetic of $\mathsf{OA}$ and the second‑order comprehension schema; in particular it does not depend on the choice of a regular encoding of any set‑theoretic universe.

Because $\mathsf{OA}$ contains the axiom of infinity, the smallest infinite ordinal $\omega$ is definable as the class of all finite ordinals:
\[
\operatorname{Nat}(\alpha)\;:=\;
\forall\beta{<}\alpha\,(\beta=0\lor\exists\gamma{<}\beta\,\forall\delta{<}\beta\,(\delta\leq\gamma))\;\land\;
\exists\beta{<}\alpha\,\forall\gamma{<}\alpha\,(\gamma\leq\beta).
\]
On $\omega$ we have the usual operations of addition and multiplication, which coincide with the restrictions of $+$ and $\times$.  
Using the class comprehension schema, every arithmetical predicate on $\omega$ is definable as a class; in particular we can carry out the primitive recursive definitions that follow.

By the results of Appendix A, there is a definable Gödel pairing function $\operatorname{Pair}(\alpha,\beta)$ that is strictly increasing in both arguments and yields a bijection between $\mathbf{Ord}$ and $\mathbf{Ord}\times\mathbf{Ord}$.  
We encode finite sequences of ordinals by iterating $\operatorname{Pair}$:
\[
\begin{aligned}
\langle \alpha_0\rangle &:= \alpha_0,\\
\langle \alpha_0,\alpha_1,\dots,\alpha_k\rangle &:=
\operatorname{Pair}(\alpha_0,\langle \alpha_1,\dots,\alpha_k\rangle).
\end{aligned}
\]
For a sequence $s=\langle \alpha_0,\dots,\alpha_{k-1}\rangle$ we write $\operatorname{len}(s)=k$ and $(s)_i=\alpha_i$ for $i<k$.  
Both the relation “$s$ is a sequence of length $k$” and the projection $(s)_i$ are definable by first‑order formulas over $\langle <,+,\times\rangle$, uniformly in $i,k$.  
We also fix a definable function $s*\alpha$ that appends $\alpha$ to the sequence $s$; it can be taken as $\operatorname{Pair}(s,\alpha)$ together with a suitable marker for the new length.

Every assignment of values to finitely many variables will be represented by a sequence.  
If $a$ is a sequence of length $\ell$, we write $a(i)$ for the value of the variable $v_i$ in $a$:
\[
a(i) := \begin{cases}
(s)_j &\text{if there exists }j<\ell\text{ such that }(s)_j\text{ codes the pair }(i,\alpha),\\
0 &\text{otherwise.}
\end{cases}
\]
For technical simplicity we adopt a slightly different representation that is better suited for the recursive definitions to come.  
Let a \emph{numbered assignment} be a sequence
\[
a = \langle i_1,\alpha_1,\; i_2,\alpha_2,\; \dots,\; i_r,\alpha_r\rangle
\]
where $i_1<i_2<\dots <i_r$ are variable indices and $\alpha_1,\dots,\alpha_r$ are the corresponding ordinals.  
We say $a$ is \emph{adequate for} a formula $\varphi$ if every free variable of $\varphi$ occurs among the $i_j$'s.  
The value $\operatorname{val}(a,i)$ is defined to be $\alpha_j$ if $i=i_j$ for some $j$, and $0$ otherwise.  
If $a$ is adequate for $\varphi$ and $x_1,\dots,x_k$ are ordinals, we write $a[v_{i_1}/x_1,\dots,v_{i_k}/x_k]$ for the assignment obtained by updating the values of $v_{i_1},\dots,v_{i_k}$ to $x_1,\dots,x_k$ (inserting new pairs if necessary).  
All these manipulations are primitive recursive in the codes and therefore definable in $\mathsf{OA}$.

The first‑order language $\mathcal{L}$ contains:
\begin{itemize}
\item logical symbols $\neg,\land,\exists$;
\item the binary relation symbols $<,=, +, \times$;
\item for each $j=1,\dots,m$ a unary predicate symbol $\dot{P}_j$ (intended to be interpreted by the class $P_j$);
\item variables $v_1,v_2,\dots$ (we use $v_i$ with $i\in\omega\setminus\{0\}$ for simplicity).
\end{itemize}
We assign distinct natural numbers to the primitive symbols:
\[
\begin{array}{c|c}
\text{symbol} & \text{code} \\ \hline
\neg & 0 \\
\land & 1 \\
\exists & 2 \\
< & 3 \\
= & 4 \\
+ & 5 \\
\times & 6 \\
v_i\;(i\geq 1) & 7+i \\
\dot{P}_j\;(1\leq j\leq m) & 7+\omega+j
\end{array}
\]
Finite strings of symbols are then coded as sequences of natural numbers, which in turn are coded as single ordinals via the pairing function.  
We write $\ulcorner\tau\urcorner$ for the code of a term $\tau$ and $\ulcorner\varphi\urcorner$ for the code of a formula $\varphi$.  
The following syntactic predicates and functions are primitive recursive in the codes and therefore definable over $\omega$ by arithmetical formulas (hence definable as classes in $\mathsf{OA}$):
\begin{itemize}
\item $\operatorname{Var}(n)$ : $n$ codes a variable $v_i$.
\item $\operatorname{Term}(n)$ : $n$ codes a term (since there are no function symbols besides the variables, this simply means $n$ codes a variable).
\item $\operatorname{AtFmla}(n)$ : $n$ codes an atomic formula, i.e., one of the forms $v_i<v_j$, $v_i=v_j$, $v_i+v_j=v_k$, $v_i\times v_j=v_k$, or $v_i\in\dot{P}_j$.
\item $\operatorname{Fmla}(n)$ : $n$ codes a first‑order formula, defined by the usual closure under $\neg,\land,\exists$.
\item $\operatorname{Neg}(n)$ : the code of $\neg\varphi$ given $\ulcorner\varphi\urcorner$.
\item $\operatorname{Conj}(n,m)$ : the code of $\varphi\land\psi$.
\item $\operatorname{Ex}(i,n)$ : the code of $\exists v_i\,\varphi$.
\item $\operatorname{FreeVar}(n)$ : the (finite) set of free variable indices of the formula coded by $n$.
\item $\operatorname{Subst}(n,i,\ulcorner\tau\urcorner)$ : the code of the formula obtained by substituting term $\tau$ for $v_i$; in our restricted language terms are only variables, so substitution is just renaming.
\end{itemize}
We shall use these predicates freely, knowing that they can be written as formulas in the language of $\mathsf{OA}$.

Let $a$ be a numbered assignment and $n$ an atomic formula code.  
We define the formula $\operatorname{AtSat}(n,a)$ (with free second‑order parameters $\vec{P}$) by cases:
\[
\operatorname{AtSat}(n,a) := 
\begin{cases}
\operatorname{val}(a,i) < \operatorname{val}(a,j), & \text{if } n = \ulcorner v_i < v_j\urcorner,\\[2pt]
\operatorname{val}(a,i) = \operatorname{val}(a,j), & \text{if } n = \ulcorner v_i = v_j\urcorner,\\[2pt]
\operatorname{val}(a,i) + \operatorname{val}(a,j) = \operatorname{val}(a,k), & \text{if } n = \ulcorner v_i + v_j = v_k\urcorner,\\[2pt]
\operatorname{val}(a,i) \times \operatorname{val}(a,j) = \operatorname{val}(a,k), & \text{if } n = \ulcorner v_i \times v_j = v_k\urcorner,\\[2pt]
\operatorname{val}(a,i) \in P_j, & \text{if } n = \ulcorner v_i \in \dot{P}_j\urcorner,\\[2pt]
\text{false}, & \text{otherwise (ill‑formed code)}.
\end{cases}
\]
Because $<$, $+$, $\times$ and the membership in the classes $P_j$ are directly expressible in the second‑order language, $\operatorname{AtSat}(n,a)$ is a legitimate formula of our second‑order language with parameters $\vec{P}$.

A \emph{satisfaction class} (relative to $\vec{P}$) is a second‑order object $S\subseteq\mathbf{Ord}$ that satisfies the Tarskian inductive conditions.  
Formally we define the formula $\operatorname{SatClass}(S)$ as the conjunction of the following four statements:

\begin{enumerate}
\item \textbf{Atomic.}
$\forall n\,\forall a\;\bigl(\operatorname{AtFmla}(n)\;\to\;
\bigl(\operatorname{Pair}(n,a)\in S \;\leftrightarrow\;\operatorname{AtSat}(n,a)\bigr)\bigr).$

\item \textbf{Negation.}
$\forall n\,\forall a\;\bigl(\operatorname{Fmla}(n)\;\to\;
\bigl(\operatorname{Pair}(\operatorname{Neg}(n),a)\in S \;\leftrightarrow\;
\operatorname{Pair}(n,a)\notin S\bigr)\bigr).$

\item \textbf{Conjunction.}
$\forall n,m\,\forall a\;\bigl(\operatorname{Fmla}(n)\land\operatorname{Fmla}(m)\;\to\;
\bigl(\operatorname{Pair}(\operatorname{Conj}(n,m),a)\in S \;\leftrightarrow\;
(\operatorname{Pair}(n,a)\in S\land\operatorname{Pair}(m,a)\in S)\bigr)\bigr).$

\item \textbf{Existential quantification.}
$\forall i\,\forall n\,\forall a\;\bigl(\operatorname{Fmla}(n)\;\to\;
\bigl(\operatorname{Pair}(\operatorname{Ex}(i,n),a)\in S \;\leftrightarrow\;
\exists y\;\operatorname{Pair}(n,a[v_i/y])\in S\bigr)\bigr).$
\end{enumerate}
Here $a[v_i/y]$ denotes the numbered assignment obtained from $a$ by updating the value for $v_i$ to $y$ (if $v_i$ is not already present, the pair $(i,y)$ is appended).  
This operation is primitive recursive in the codes, hence definable.

Because the second‑order comprehension schema of $\mathsf{OA}$ guarantees the existence of classes defined by arbitrary formulas, the expression “$\exists S\,\operatorname{SatClass}(S)$” is a well‑formed formula of our second‑order language (with parameters $\vec{P}$).

Now we can define the truth predicate.  
Let $\varphi$ be a first‑order formula with free variables among $v_{i_1},\dots,v_{i_k}$ (listed in increasing order of indices) and class parameters among $\vec{P}$.  
For arbitrary ordinals $x_1,\dots,x_k$ we set:
\[
\operatorname{Sat}(\ulcorner\varphi\urcorner,x_1,\dots,x_k) \;:=\;
\exists S\;\Bigl(\operatorname{SatClass}(S)\;\land\;
\operatorname{Pair}\bigl(\ulcorner\varphi\urcorner,\;
a_0[v_{i_1}/x_1,\dots,v_{i_k}/x_k]\bigr)\in S\Bigr),
\]
where $a_0$ is any fixed empty assignment (e.g., the code of the empty sequence).  
Because $a_0[v_{i_1}/x_1,\dots,v_{i_k}/x_k]$ depends definably on $x_1,\dots,x_k$, the right‑hand side is a formula of the monadic second‑order language with parameters $\vec{P}$.

\begin{theorem}[Tarski’s equivalence]
For every first‑order formula $\varphi(v_{i_1},\dots,v_{i_k},\vec{P})$ the following is provable in $\mathsf{OA}$:
\[
\operatorname{Sat}(\ulcorner\varphi\urcorner,x_1,\dots,x_k)\;\longleftrightarrow\;
\varphi(x_1,\dots,x_k,\vec{P}).
\]
\end{theorem}
\begin{proof}[Proof sketch]
The proof proceeds by induction on the complexity of $\varphi$.  
A more general claim is proved: for every $\varphi$ and every numbered assignment $a$ adequate for $\varphi$,
\[
\exists S\,\bigl(\operatorname{SatClass}(S)\land \operatorname{Pair}(\ulcorner\varphi\urcorner,a)\in S\bigr)
\;\longleftrightarrow\; \varphi[a],
\]
where $\varphi[a]$ is the result of replacing each free variable $v_i$ in $\varphi$ by the ordinal $\operatorname{val}(a,i)$ and each $\dot{P}_j$ by the class $P_j$.

\textbf{Existence.}
By the second‑order comprehension schema we can form the class
\[
S_0 := \{\operatorname{Pair}(\ulcorner\psi\urcorner,a) \mid
\psi\text{ a subformula of }\varphi,\; a\text{ adequate},\; \psi[a]\text{ true}\}.
\]
One verifies by induction on $\psi$ that $S_0$ satisfies the four clauses of $\operatorname{SatClass}$ when restricted to subformulas of $\varphi$.  
Extending $S_0$ arbitrarily (e.g., by putting all other pairs not in $S_0$) yields a full satisfaction class $S$ that coincides with $S_0$ on the subformulas of $\varphi$, because the truth of a formula depends only on the truth of its proper subformulas.  
Hence the left‑hand side of the equivalence holds if $\varphi[a]$ is true.

\textbf{Uniqueness on subformulas.}
If $S_1$ and $S_2$ are two satisfaction classes, a straightforward induction on $\psi$ shows that
\[
\operatorname{Pair}(\ulcorner\psi\urcorner,a)\in S_1\;\longleftrightarrow\;
\operatorname{Pair}(\ulcorner\psi\urcorner,a)\in S_2
\]
for every $\psi$ that is a subformula of $\varphi$ and every assignment $a$.  
The atomic clause fixes the values of atomic formulas, and the remaining clauses force the agreement to propagate under $\neg,\land,\exists$.  
Consequently, if there exists \emph{any} satisfaction class containing $\operatorname{Pair}(\ulcorner\varphi\urcorner,a)$, then $\varphi[a]$ must be true, because the “intended” satisfaction class $S_0$ built from the true subformulas must contain that pair as well.

Together, the two directions yield the required equivalence.  
The induction is completely constructive and can be formalized in $\mathsf{OA}$ using the induction principle on the well‑founded subformula relation, which is available because the formulas are coded by natural numbers and the subformula relation is a well‑founded relation on $\omega$.
\end{proof}

Thus, we have obtained a truth predicate $\operatorname{Sat}$ in monadic second‑order logic for all first‑order formulas, with the required class parameters $\vec{P}$ fixed.

\bibliographystyle{plain}
\bibliography{main}

\end{document}